\documentclass{amsart}
\usepackage{amssymb,amsmath,amsthm}
\usepackage{enumerate,color}
\usepackage{enumitem}
\usepackage{xcolor}
\usepackage[pdftex]{graphicx}
\newcommand{\dashint}{\int\!\!\!\!\!-}
\def\Xint#1{\mathchoice
{\XXint\displaystyle\textstyle{#1}}%
{\XXint\textstyle\scriptstyle{#1}}%
{\XXint\scriptstyle\scriptscriptstyle{#1}}%
{\XXint\scriptscriptstyle\scriptscriptstyle{#1}}%
\!\int}
\def\XXint#1#2#3{{\setbox0=\hbox{$#1{#2#3}{\int}$ }
\vcenter{\hbox{$#2#3$ }}\kern-.6\wd0}}

\def\dashint{\Xint-}

\usepackage[colorlinks=true, pdfstartview=FitV, linkcolor=blue, citecolor=blue, urlcolor=blue,pagebackref=true]{hyperref}
\usepackage{esint}
\usepackage{mathrsfs}

\usepackage{mathtools}

\newtheorem{theorem}{Theorem}[section]

\newtheorem{proposition}[theorem]{Proposition}

\theoremstyle{definition}
\newtheorem{definition}[theorem]{Definition}
\newtheorem{remark}[theorem]{Remark}

\numberwithin{equation}{section}

\newcommand{\loc}{{\rm loc}}

\DeclareMathOperator{\supp}{\operatorname{supp}}

\begin{document}
\title[Absence of local anomalous dissipation]{Absence of local anomalous dissipation and local energy balance in 2D incompressible flows away from the boundary}

\author[T. Barker et al.]{Tobias Barker}
\author[]{Milton C. Lopes Filho}
\author[]{Helena J. Nussenzveig Lopes}
\date{\today}
\begin{abstract}
For the 2D Navier-Stokes equations with no-slip boundary condition, we consider the issue of whether anomalous dissipation away from the boundary vanishes. In particular, we show that such vanishing occurs if $u^{\nu}$ is uniformily bounded in the Onsager supercritical space $L^{1+}_{t}L^{\infty}_{x,\loc}$ with appropriate bounds on the initial conditions. Our method involves arguments  from \cite{AD23} and \cite{CW23} involving localization via modulation, together with vorticity energy type estimates inspired by \cite{CFLS16} and estimates involving $L^2$-based structure functions inspired by \cite{DP25dissconc}.

Next we show that the aforementioned setting produces convergence to an Euler solution with its \textit{large scale approximation} satisfying a local energy balance equation. Notably, we do not assume any uniform-in-viscosity bounds on the pressure. The large scale approximation has been introduced in \cite{PGLR18} in the context of partial regularity of the 3D Navier-Stokes equations, yet to the best of our knowledge this is the first time it has been considered in the context of inviscid limits.
\end{abstract}

\maketitle

\tableofcontents

\section{Introduction} \label{s:intro}
For $\Omega\subseteq\mathbb{R}^2,\,\mathbb{R}^3$ and $\nu>0$, consider a solution $u^{\nu}$ the Navier-Stokes equations
 \begin{equation}\label{eq:NSE}
\begin{aligned}
	&\partial_{t}u^{\nu}-\nu\Delta u^{\nu}+u^{\nu}\cdot\nabla u^{\nu}+\nabla p^{\nu}=0,
	&\textrm{div}\,u^{\nu}=0,\qquad u^{\nu}(\cdot,0)=u^{\nu}_0.
\end{aligned}
\end{equation}
In the case where $\partial\Omega$ is non-empty, we impose the no-slip boundary condition
\begin{equation}\label{eq:noslip}
u^{\nu}|_{\partial\Omega}=0.
\end{equation}
A core feature of turbulence is the so called \textit{zeroth-law}
\begin{equation}\label{eq:zerolaw}
\lim_{\nu\rightarrow 0}\nu<|\nabla u^{\nu}|^2>=\epsilon>0,
\end{equation}
where $<\cdot>$ denotes a certain averaging process. For the case $\Omega=\mathbb{T}^3$, assuming solution bounds from the K41 theory of turbulence, \cite{K41.1.1}-\cite{K41.2.1} imply that $u^{\nu}$ converges up to subsequence to a weak solution to the Euler equations $u^{E}$ \cite{CG12}. This, together with the energy inequality for $u^{\nu}$
$$\|u^{\nu}(\cdot,t)\|_{L^{2}(\Omega)}^2+2\nu\int\limits_{0}^{t}\int\limits_{\Omega}|\nabla u^{\nu}|^2 dxds\leq \|u_{0}^{\nu}\|_{L^{2}(\Omega)}^2 $$ and the zeroth-law, then implies the non-conservation of the kinetic energy of the  above limiting Euler solution
$$t\rightarrow \|u^{E}(\cdot,t)\|_{L^{2}(\Omega)}^2.$$
Such considerations lead to the question:
\begin{itemize}
\item []\textit{What is the regularity threshold at which weak solutions of the Euler equations conserve kinetic energy?}
\end{itemize}
This is known in the mathematical community as `Onsager's conjecture', with the simplest version stating that solutions to Euler conserve energy if and only if they belong to $C^{\alpha}_{x,t}$ with $\alpha>\tfrac{1}{3}.$ Onsager's conjecture has attracted significant attention, which we cannot pretend to allude to. For $\Omega=\mathbb{T}^3$, the rigidity side of Onsager's conjecture was addressed in so-called Onsager (sub)critical spaces in \cite{E94} and \cite{CET94}, which culmulated in the energy conservation result for Euler solutions in the Onsager critical space $L^{3}_{t}{B}^{\frac{1}{3}}_{3,\mathbb{N}}$ in \cite{CCFS08} that is somewhat sharp. For further results on energy conservation in Onsager critical spaces, see \cite{DI24} and \cite{DIN24}, for example.

 The flexibility side of Onsager's conjecture has been investigated using the powerful machinery of `convex integration'. This was systematically introduced in \cite{DeSz09} and \cite{DeSz13} in the context of fluid dynamics. A series of convex integration works \cite{DeSz14}, \cite{BDeISz15}-\cite{BDeSz16} led to the full resolution of Onsager's conjecture from the flexibility side in \cite{I18}. For further results on Onsager's conjecture, see \cite{BDeSzV19}, \cite{BMNV23}, \cite{NV23}, \cite{GR24}, \cite{GKN25}, for example. For a thorough account of convex integration in fluid dynamics, we recommend the excellent surveys in \cite{BV19survey}-\cite{BV21survey} and \cite{DeSz22survey}, for example.
 
As mentioned previously, predictions of turbulence concern \eqref{eq:zerolaw} for Navier-Stokes equations and has implications for Euler solutions achieved in the vanishing viscosity limit. It is therefore most relevant to ask
\begin{itemize}
\item[]\textit{Do Euler solutions obtained in the vanishing viscosity limit (`physically realizable solutions') conserve kinetic energy or not?}
\end{itemize}
In $\mathbb{T}^2$, the work \cite{CFLS16} showed that in settings below Onsager rigidity thresholds (`Onsager supercritical'), anomalous dissipation \eqref{eq:zerolaw} vanishes and the corresponding physically realizable Euler solutions conserve the kinetic energy. Intriguingly, recently the existence of two dimensional convex integration solutions with vorticity integrability greater than 1 was shown in \cite{BCK24}. This convex integration solution is in the same Onsager supercritical regime as the result in \cite{CFLS16} establishing energy conservation for physically realizable solutions, demonstrating an important distinction between these two notions of solutions. Inspired by such issues, in this paper we investigate the following \textit{localized} questions in physical settings with a boundary present.
\begin{enumerate}[label={\bf {Q\arabic*.}}, ref={\bf {Q\arabic*}}]
\item  For a smooth bounded domain $\Omega\subset \mathbb{R}^2$ and for weak Leray-Hopf solutions\footnote{For the definition of weak Leray-Hopf solution to \eqref{eq:NSE}-\eqref{eq:noslip}, see Definition \ref{Lerays}.} $u^{\nu}$ with no-slip boundary condition,  can local anomalous dissipation be ruled out away from the boundary in Onsager supercritical regimes? \label{Q1}
\medskip
\item  In the setting of the previous question, what implications does local absence of anomalous dissipation have for local energy balance for the associated physically realizable solution $u^{E}$? \label{Q2}
\end{enumerate}

\subsection{Main results}
\subsubsection{Conditions on the initial data}
Let $\Omega\subset\mathbb{R}^2$ be a bounded domain with sufficiently smooth boundary and $1<p<\infty$.
Let $(u^{\nu}_{0})_{\nu>0}$ be a family of {\em initial velocities} in\footnote{ See the section `\ref{ss:Function spaces} Function spaces and norms' for the definition of $W^{1,2}_{0,\sigma}(\Omega)$.} $W^{1,2}_{0,\sigma}(\Omega)$ such that 
\begin{equation}\label{eq:L2iduniform} \tag{\textbf{Condition A}}
\sup_{\nu>0}\|u^{\nu}_{0}\|_{L^{2}(\Omega)}=J<\infty. 
\end{equation}
Throughout we will use the following notation: for $\Omega\subset \mathbb{R}^2$ being a bounded set we define $\Omega(r)$ to be the compact set
\begin{equation}\label{eq:domaindistfromboundary}
\Omega(r):=\{x\in\Omega: \textrm{dist}\,(x,\partial\Omega)\geq r\}.
\end{equation}
 In certain statements, we will also suppose that, for all $r>0$, there exists $N(r)>0$ such that the {\em initial vorticities} $\omega_0^{\nu}=-\partial_{2}(u_0^{\nu})_{1}+\partial_{1}(u_0^{\nu})_{2}$ satisfy the uniform bound
 \begin{equation}\label{eq:vortLpuniform} \tag{\textbf{Condition B}}
 \sup_{\nu>0}\|\omega^{\nu}_{0}\|_{L^{p}(\Omega(r))}= M(r)<\infty.
 \end{equation}
\subsubsection{Statement of results}

In the theorem below, we address \ref{Q1}.
\begin{theorem}\label{thm:nolocaanondis}
Let $\Omega\subset\mathbb{R}^2$ be a bounded domain with sufficiently smooth boundary and $1<p<\infty$. Let $(u^{\nu}_{0})_{\nu>0}$ be a family of initial data in $W^{1,2}_{0,\sigma}(\Omega)$ satisfying the uniform bounds in \ref{eq:L2iduniform},  \ref{eq:vortLpuniform}.

 For each $\nu>0$, let $u^{\nu}:\Omega\times (0,\infty)\rightarrow \mathbb{R}^2$ be the unique weak Leray-Hopf solution to \eqref{eq:NSE}-\eqref{eq:noslip} with initial data $u_{0}^{\nu}$.

 Suppose that there exists $\varepsilon>0$ and $T>0$ such that, for all $r>0$, there exists $N(r)>0$ with 
 \begin{equation}\label{eq:L1plusinfinity} \tag{\textbf{Condition C}}
 \sup_{\nu>0}\|u^{\nu}\|_{L^{1+\varepsilon}(0,T; L^{\infty}(\Omega(r)))}=N(r)<\infty.
 \end{equation}
 Then, for every compact set $K\subset\Omega$, there exists $t_{K}^{*}=t_{K}^{*}(K, p, \Omega, T,\varepsilon, N(\tfrac{1}{2}\textrm{dist}\,(K,\partial\Omega)))\in (0,T]$ such that
 \begin{equation}\label{eq:thm1conclusionvortbound}
 \sup_{\nu>0}\|\omega^{\nu}\|_{L^{\infty}(0,t_{K}^{*}; L^{p}(K))}<\infty\quad\textrm{and}
 \end{equation}
 \begin{equation}\label{eq:thm1conclusion}
 \lim_{\nu\rightarrow 0}\nu\int\limits_{0}^{t_{K}^{*}}\int\limits_{K} |\nabla u^{\nu}|^2 dxdt=0.
 \end{equation}
\end{theorem}
We address \ref{Q2} below for weak Euler solutions\footnote{See Definition \ref{def:weaksol}.} arising in the vanishing viscosity limit.
We address this in terms of a suitably defined local Leray projection/large scale approximation of the Euler solution, which will be introduced in \eqref{eq:locallerayprojectiondef}, and its local energy balance equation (see Definition \ref{def:locallerayenergybalance}).
 \begin{theorem}\label{thm:localenergybalance}
 Suppose all the assumptions in Theorem \ref{thm:nolocaanondis}  are valid. Then there exists a subsequence of $(u^{\nu})_{\nu>0}$, which we will not relabel, such that the following holds.
 
  For all $r>0$, there exists $t^*_{r}=t^*_{r}({r,\Omega,T,\varepsilon,N(r)})\in (0,T]$ such that 
 \begin{equation}\label{eq:strongL3}
 u^{\nu}\rightarrow u^{E}\quad\textrm{in}\quad L^{3}(0,t^*_{r}; L^{3}(\Omega(r))).
 \end{equation}
 In particular, $u^{E}$ is a weak solution to 2D Euler on $\Omega(r)\times(0,t^*_{r})$.
   
 Let $\varphi\in C^{\infty}_{0}(\Omega; [0,1])$ be a smooth cut-off function such that $\supp(\varphi)\Subset\Omega(r)$. Then, for any open set $\Omega_0$ satisfying 
\begin{equation}\label{eq:supportvarphi}
 \Omega_{0}\Subset \{x:\varphi(x)=1\},
 \end{equation}
 $\mathbb{P}_{\varphi}(u^{E})$ satisfies the aforementioned local energy balance equation in $\Omega_{0}\times (0,t^*_{r})$.

 \end{theorem}

 The assumption on the initial vorticities, \ref{eq:vortLpuniform}, is restrictive and, in view of recent results, see \cite{LMP21} and \cite{DP25dissconc}, it is desirable to weaken it. With this in mind, in our final result we show that we can obtain the same conclusion as in  Theorem 
 \ref{thm:localenergybalance} by substituting  \ref{eq:vortLpuniform} by a compactness condition on the family $\{u^{\nu}\}_{\nu > 0}$.
 
More precisely, we have:

\begin{theorem}\label{thm:localenergybalancecompactL3}
 Let $\Omega\subset\mathbb{R}^2$ be a bounded domain with sufficiently smooth boundary. Let $(u^{\nu}_{0})_{\nu>0}$ be a family of initial data in $W^{1,2}_{0,\sigma}(\Omega)$ satisfying \ref{eq:L2iduniform} and assume that $u^{\nu}$ satisfies \ref{eq:L1plusinfinity}. Suppose, additionally, that there exists a subsequence of $(u^{\nu})_{\nu>0}$, not relabeled, such that
\begin{equation}\label{eq:strongL3thm3}
 u^{\nu}\rightarrow u^{E}\quad\textrm{in}\quad L^{3}(0,T; L^{3}_{loc}(\Omega)).
\end{equation}

Let $\varphi\in C^{\infty}_{0}(\Omega; [0,1])$ be a smooth cut-off function such that $\supp(\varphi)\Subset\Omega(r)$. Then, for any open set $\Omega_0$ satisfying 
\begin{equation}\label{eq:supportvarphiAGAIN}
 \Omega_{0}\Subset \{x:\varphi(x)=1\},
 \end{equation}
there exists $t^*_{r}=t^*_{r}({r,\Omega,T,\varepsilon,N(r)})\in (0,T]$ such that $\mathbb{P}_{\varphi}(u^{E})$ satisfies the local energy balance equation in $\Omega_{0}\times (0,t^*_{r}).$
 \end{theorem}
 
 \subsection{Related works and strategy}
 
 \subsubsection{Absence of local anomalous dissipation}
 The starting point for our discussion is the work \cite{CFLS16} on $\mathbb{T}^2$, which has the initial vorticity $(\omega_{0}^{\nu})_{\nu>0}$ uniformily bounded in $L^{p}(\mathbb{T}^2)$ for some $1<p<\infty$. The strategy there for ruling out anomalous dissipation is to exploit the lack of vortex stretching for the 2D vorticity equation
 \begin{equation}\label{eq:vorteqn2Dintro}
 \partial_{t}\omega^{\nu}-\nu\Delta \omega^{\nu}+u^{\nu}\cdot\nabla \omega^{\nu}=0.
 \end{equation}
 In particular, in \cite{CFLS16} an $L^{p}$ energy estimate is performed giving that
 \begin{equation}\label{eq:vorticityLpineq}
 \||\omega^{\nu}|^{\frac{p}{2}}(\cdot,t)\|_{L^{2}(\mathbb{T}^2)}^2+\nu \int\limits_{0}^{t}\int\limits_{\mathbb{T}^2} |\nabla |\omega^{\nu}|^{\frac{p}{2}}|^2 dyds\leq C(p)\|\omega_{0}^{\nu}\|_{L^{p}(\mathbb{T}^2)}^p.
 \end{equation} 
 This is then combined with the Gagliardo-Nirenberg inequality and differential inequalities based on energy balance for $\|u^{\nu}(\cdot,t)\|_{L^{2}}^2$ and $\|\omega^{\nu}(\cdot,t)\|_{L^2}^2$ to infer the absence of anomalous dissipation
 $$\lim_{\nu\rightarrow 0}\nu\int\limits_{0}^{t}\int\limits_{\mathbb{T}^2} |\nabla u^{\nu}|^2 dyds=0.$$ For further results in this direction in the $\mathbb{T}^2$ setting, see \cite{LMP21}, \cite{LLJL25}, \cite{LL21}, \cite{DP25} and \cite{ELL25}, for example. See also \cite{DLP25} in the context of Hamiltonian conservation for the SQG equation.
 
 Let us now consider the setting of a smooth bounded domain $\Omega\subset\mathbb{R}^2$ and Navier-Stokes equations \eqref{eq:NSE} with no-slip on the boundary \eqref{eq:noslip}. In \cite[Theorem 3.1]{K17} trace considerations were used to show that, if $u^{\nu}$ converges to $u^{E}$ in a suitable sense, with $u^{E}$ not vanishing in the tangential direction to $\partial\Omega$, then, for $p>1$, the vorticity $\omega^{\nu}$ satisfies
 \begin{equation}\label{eq:vorticityLpblowup}
 \limsup_{v\rightarrow 0^+}\|\omega^{\nu}\|_{L^{\infty}(0,T; L^{p}(\Omega))}=\infty.
 \end{equation} 
 So, for 2D fluid flow in domains with boundaries, it seems that the previous energy approach using the vorticity equation \eqref{eq:vorteqn2Dintro} fails.
 
 In Theorem \ref{thm:nolocaanondis}, we assume $u^{\nu}$ has a uniform Onsager supercritical bound away from the boundary, namely \ref{eq:L1plusinfinity}, which grants a version of \eqref{eq:vorticityLpineq} in the interior of the domain. Recently, in the context of the inviscid limit problem, other criteria based on interior estimates have appeared in the literature, see for instance \cite{CV18}, \cite{CLLV19}, \cite{DN18}, \cite{BTW19} and \cite{SWW24}.
 
The assumed Onsager supercritical uniform bound \ref{eq:L1plusinfinity} plays the role of ensuring that, for any compact set $K\subset\Omega$, there exists $S(K)>0$ such that
\begin{equation}\label{eq:introL1+Linfinity}
\|u^{\nu}\|_{L^{1}(0,S(K); L^{\infty}(K))}\quad\textrm{ is sufficiently small},
\end{equation}
uniformly with respect to $\nu$. The condition \eqref{eq:introL1+Linfinity} gives a uniform control on the Lagrangian trajectories in the vanishing viscosity limit. Thus, one may interpret the assumption \ref{eq:L1plusinfinity} as limiting the speed at which singular structures can propagate from the boundary into a given interior region in the vanishing viscosity limit. Therefore for a short time, which degenerates as one approaches the boundary, one gains regularity estimates in the interior that are not influenced by activity on the boundary.

To exploit \eqref{eq:introL1+Linfinity} from a technical perspective, we use an idea in \cite{AD23} and \cite{CW23}, which are in the context of local Cacciopolli estimates for parabolic equations with divergence-free drifts and localized regularity criteria for the 3D Euler equations respectively.
 Namely, we use \eqref{eq:introL1+Linfinity} to suitably modulate the spatial variable of the vorticity equation \eqref{eq:vorteqn2Dintro}. As in \cite{AD23} and \cite{CW23}, \eqref{eq:introL1+Linfinity} and the modulation allows the local $L^p$ inequality for the vorticity to be transformed into a local $L^p$ inequality for the modulated vorticity, which is amenable to Gronwall's lemma and gives a version of \eqref{eq:vorticityLpineq} in the interior of the domain. With this in hand, we can then use the Gagliardo-Nirenberg inequality as in \cite{CFLS16} to conclude.
 \subsubsection{Local energy balance for physically realizable Euler solutions}
 For a  weak solution Euler solution $u^{E}$, the \textit{local} energy balance takes the form
 \begin{equation}\label{eq:localenergybalanceintro}
 \partial_{t}\Big(\frac{|u^{E}|^2}{2}\Big)+\textrm{div}\,\Big(u^{E}\Big(\frac{|u^{E}|^2}{2}+p^{E}\Big)\Big)=0
 \end{equation}
 in the sense of distributions. The study of local energy balance for weak Euler solutions was pioneered by J. Duchon and R. Robert in \cite{DR2000}. In the work \cite{CFLS16} it is shown that, for solutions $(u^{\nu}, p^{\nu})$ to the 2D Navier-Stokes equations on $\mathbb{T}^2$, with initial vorticity $\omega_{0}^{\nu}$ uniformly bounded in the Onsager supercritical space $L^{p}(\mathbb{T}^2)$ with $\frac{6}{5}<p<\frac{3}{2}$, the limit is a physically realizable Euler solution satisfying \eqref{eq:localenergybalanceintro}. To show this, the elliptic estimate in $\mathbb{T}^2$
 \begin{equation}\label{eq:CZpressureintro}
 p^{\nu}\sim |u^{\nu}|^2\quad\textrm{in}\quad L^{q}\quad 1<q<\infty
 \end{equation}
 is crucially used. For a domain $\Omega$ with boundary and if $u^{\nu}$ satisfies the no-slip boundary condition, the estimate \eqref{eq:CZpressureintro} is unavailable and analogous estimates are even false for solutions of the Stokes system \cite{KS02}. Other than certain symmetric settings \cite{LLMT08}, little appears to be known about the behavior of $p^{\nu}$ in the vanishing viscosity limit, in the case of boundaries and under the no-slip boundary condition. In this setting the  previous results on energy balance and anomalous dissipation for physically realizable solutions  (e.g \cite{BTW19} and \cite{DN18}) are based on additional assumptions for the pressure $p^{\nu}$ which are uniform in viscosity. 
 
Due to the aforementioned issues with the pressure in the setting with boundaries, we consider \ref{Q2} for the local Leray projection/large scale approximation of a physically realizable Euler solution $u^{E}$. This is formally\footnote{See \eqref{eq:locallerayprojectiondef} for the definition.} given by
\begin{equation}\label{eq:largescaleapproxintro}
\mathbb{P}_{\varphi}(f):=\nabla^{\perp} \Delta_{2D}^{-1}(\varphi\nabla^{\perp}\cdot f),
\end{equation}
where $\varphi$ is a suitable cut-off function, $\Delta_{2D}^{-1}$ is the operator given by convolution with the full plane Green's function $\frac{1}{2\pi}\log(|x|)$, and where $$\nabla^{\perp}=\begin{bmatrix} -\partial_{2}\\
\partial_{1} \end{bmatrix}.$$
This was introduced in \cite{PGLR18} and revisited in \cite{K23} in the context of the partial regularity theory for the 3D Navier-Stokes equations, yet has not been considered in the context of inviscid limits. The advantage is that, for smooth enough $p^{\nu}$, one has
$$\mathbb{P}_{\varphi}(\nabla p^{\nu}):=\nabla^{\perp} \Delta_{2D}^{-1}(\varphi\nabla^{\perp}\cdot \nabla p^{\nu})=0. $$
This ultimately means $\mathbb{P}_{\varphi}(u^{\nu})$ satisfies an equation not involving the pressure $p^{\nu}$, as observed in \cite{PGLR18} and \cite{K23}. Using the 2D analogue of the equation derived in \cite{K23}, we show that, in the setting of Theorem \ref{thm:localenergybalance}, the absence of local anomalous dissipation results in $\mathbb{P}_{\varphi}(u^E)$ satisfying a local energy balance equation. 

We achieve Theorem \ref{thm:localenergybalancecompactL3} by using  the aforementioned localized energy estimate for the modulated vorticity and arguments from \cite[Proposition 3.2, 4.1]{DP25dissconc}\footnote{These arguments from \cite{DP25dissconc} are in the context of the torus $\mathbb{T}^2$.}, which enable the dissipation to be estimated in terms of a $L^2$-based structure function of $u^{\nu}$  in the interior of the domain. Once we show that the dissipation vanishes locally, we then argue in the same way as Theorem \ref{thm:localenergybalance} to infer that $\mathbb{P}_{\varphi}(u^E)$ satisfies a local energy balance.
 \subsection{Further remarks}
 \begin{remark}\label{rmk:additionalforcing}[On forced extensions]
 It can be readily verified that the conclusion of Theorem \ref{thm:nolocaanondis} holds true with an additional forcing $F^{\nu}:\Omega\times (0,T)\rightarrow \mathbb{R}^2$ that satisfies the uniform bounds 
 \begin{equation}\label{eq:L2forceuniform} \tag{\textbf{Condition D}}
\sup_{\nu>0}\|F^{\nu}\|_{L^{1}(0,T;L^{2}(\Omega))}<\infty 
\end{equation}
and for all $r>0$
\begin{equation}\label{eq:forcecurluniform} \tag{\textbf{Condition E}}
\sup_{\nu>0}\|\nabla^{\perp}\cdot F^{\nu}\|_{L^{1}(0,T; L^p(\Omega(r)))}=Q(r)<\infty 
\end{equation}
The condition \ref{eq:forcecurluniform} changes the estimate \eqref{eq:aprioriest1anyp} to 
\begin{equation}\label{eq:aprioriest1anypforce}
\begin{split}
&\sup_{t\in (0,T)}\int\limits_{B(\frac{3R_0}{8})} |\omega^{\nu}(x,t)|^p dx+\nu\int\limits_{0}^{T}\int\limits_{B(\frac{3R_0}{8})} |\nabla|\omega^{\nu}|^{\frac{p}{2}}|^2 dxds\\
&\leq C(p)\Big(\int\limits_{B(R_0)} |\omega_{0}^{\nu}(x)|^p dx+\|\nabla^{\perp}\cdot F^{\nu}\|_{L^{1}(0,T; L^p(B(R_0)))}^p+\frac{\nu}{R_0^2}\int\limits_{0}^{T} \int\limits_{B(R_0)} |\omega^\nu|^p dxds\Big)
\end{split}
\end{equation}
with the  estimate \eqref{eq:aprioriest1} changing according to H\"{o}lder's inequality  applied to \eqref{eq:aprioriest1anypforce}. The forcing is also incorporated into the energy estimate \eqref{eq:energyequality}, giving
\begin{equation}\label{eq:energyequalityforce}
\sup_{t\in (0,T)}\|u^{\nu}(\cdot,t)\|_{L^{2}(\Omega)}^2+\nu\int\limits_{0}^{T}\int\limits_{\Omega}|\nabla u^{\nu}(x,s)|^2 dxds\leq C\|u_{0}^{\nu}\|_{L^{2}(\Omega)}^2+C\|F^{\nu}\|_{L^{1}(0,T;L^{2}(\Omega))}^2. 
\end{equation}
Using \ref{eq:L2iduniform}, \ref{eq:vortLpuniform} and \ref{eq:L2forceuniform}, \ref{eq:forcecurluniform}, we have that the estimates \eqref{eq:aprioriest1anypforce}-\eqref{eq:energyequalityforce} are uniform in $\nu$. The rest of the proof is unchanged compared to the (unforced) Theorem \ref{thm:nolocaanondis}.
 \end{remark}
 \begin{remark}\label{rmk:uniformboundassumption}[On the Onsager supercritical bound \ref{eq:L1plusinfinity}]
 We do not know if the assumption on the Navier-Stokes solutions \ref{eq:L1plusinfinity} holds unconditionally.
 Such a uniform bound is known to hold for the Stokes system in a bounded $C^3$ domain \cite{AG13} or an exterior domain in $\mathbb{R}^2$ \cite{A21} provided the initial data is uniformily bounded in $L^{\infty}(\Omega)$.
 \end{remark}
\begin{remark}\label{rmk:comparisonwithSWW24}[Regarding strong convergence of vorticity]
In \cite{SWW24}, strong convergence of the vorticity is shown under the assumption \ref{eq:L2iduniform}, together with the assumption that for $p\in (2,\infty)$ that $\omega_0^{\nu}\rightarrow \omega_0$ in $L^{p}_{loc}(\Omega)$ and that for every compact set $K\subset\Omega$
\begin{equation}\label{eq:SWWvortuniformilybounded}
\sup_{\nu>0}\|\omega^{\nu}\|_{L^{\infty}(0,T; L^{p}(K))}<\infty.
\end{equation}
For the case $p\in (2,\infty)$ and $\omega_0^{\nu}\rightarrow \omega_0$ in $L^{p}_{loc}(\Omega)$, we see that the assumption \ref{eq:L1plusinfinity} in Theorem \ref{thm:nolocaanondis} is strictly weaker than \eqref{eq:SWWvortuniformilybounded} by the Sobolev embedding theorem. 
In Theorem \ref{thm:nolocaanondis}, we show that these weaker assumptions imply the uniform vorticity bound \eqref{eq:thm1conclusionvortbound}. Consequently, in the setting of Theorem \ref{thm:nolocaanondis} with $p\in (2,\infty)$, we expect to also get the strong vorticity convergence proven in \cite{SWW24}.
\end{remark} 
\begin{remark}\label{rmk:localenergy balance}[Conditions giving the local energy balance and the 3D case]
If $\Omega\subset \mathbb{R}^3$,
\begin{equation}\label{rmk:strongL3loc}
u^{\nu}\rightarrow u^{E}\quad\textrm{in}\quad L^{3}(0,T; L^{3}_{loc}(\Omega))
\end{equation} 
and
\begin{equation}\label{rmk:absenceanon}
\lim_{\nu\rightarrow 0} \nu\int\limits_{0}^{T}\int\limits_{\Omega} \varphi |\nabla u^{\nu}|^2 dyds=0\quad\forall \varphi\in C^{\infty}_{0}(\Omega; [0,1]),
\end{equation}
then a similar conclusion to Theorem \ref{thm:localenergybalancecompactL3} holds. This can be achieved by utilizing the equations in 3D for the local Leray projector from \cite{K23}.
\end{remark}
\subsection{Notation}
\subsubsection{Universal constants, vectors and domains}
When a constant depends on a parameter (for example $p$), we will sometimes denote its dependence by $c(p)$. In this paper, universal constants will sometimes be denoted by $C_{univ}$ and $C$. These may change from line to line unless otherwise specified. Sometimes we will write $X\lesssim Y\leq Z$ to mean that there exists a positive universal constant $C$ such that $X\leq CY\leq CZ.$ Sometimes in this paper, we will write $X\lesssim_{p} Y$. This means that there exists a positive constant $C(p)$ such that $X\leq C(p)Y.$

For $x\in\mathbb{R}^{2}$ and $r>0$, 
we write $B(x,r):= \{y\in\mathbb{R}^2; |x-y|<r\}$. For $x=0$, we write $B(r)$ instead.
For $a,\, b\in\mathbb{R}^2$, we write $(a\otimes b)_{\alpha\beta}=a_\alpha b_\beta$, and for $A,\, B\in M_2(\mathbb{R})$, $A:B=A_{\alpha\beta}B_{\alpha\beta}$. Here and in the whole paper we use Einstein's convention on repeated indices.  For $F:\Omega\subseteq\mathbb{R}^2\rightarrow\mathbb{R}^2$, we define $\nabla F\in M_{2}(\mathbb{R})$ by $(\nabla F(x))_{\alpha\beta}:= \partial_{\beta} F_{\alpha}$. 

\subsubsection{Function spaces and norms} \label{ss:Function spaces}

For all $\mathcal{O}\subseteq \mathbb{R}^2$ open, $k\in \mathbb{N}$ and $1\leq p\leq \infty$, $W^{k,p}(\mathcal{O})$ denotes the Sobolev space with differentiability $k$ and integrability $p$. Furthermore, $L^{2}_{\sigma}(\mathcal{O})$ and $W^{1,2}_{0,\sigma}(\mathcal{O})$ denote the closure of smooth compactly supported divergence-free test functions, with respect to the $L^{2}(\mathcal{O})$ and $W^{1,2}(\mathcal{O})$ norms respectively.
  
If $X$ is a Banach space with norm $\|\cdot\|_{X}$, then $L^{s}(a,b;X)$, with $a<b$ and $s\in[1,\infty)$,  will denote the usual Banach space of strongly measurable $X$-valued functions $f(t)$ on $(a,b)$ such that
$$\|f\|_{L^{s}(a,b;X)}:=\left(\int\limits_{a}^{b}\|f(t)\|_{X}^{s}dt\right)^{1/s}<+\infty.$$ 
The usual modification is made if $s=\infty$. Let $C([a,b]; X)$ denote the space of continuous $X$ valued functions on $[a,b]$ with usual norm. In addition, let $C_{w}([a,b]; X)$ denote the space of $X$ valued functions, which are continuous from $[a,b]$ to the weak topology of $X$. 

Let $a,b\in\mathbb{R}$. Sometimes we will denote $L^{p}(a,b; L^{q}(\mathcal{O}))$ and $L^{p}(a,b; W^{k,q}(\mathcal{O}))$ by $L^{p}_{t}L^{q}_{x}(\mathcal{O}\times (a,b))$ and $L^{p}_{t}W^{k,q}_{x}(\mathcal{O}\times (a,b)).$ For the case $p=q$ we will often write $L^{p}_{x,t}(\mathcal{O}\times (a,b))$ to denote $L^{p}(a,b; L^{p}(\mathcal{O}))$.

\subsubsection{Solution concepts}
\begin{definition}\label{Lerays}[Leray-Hopf solution]
Let $\Omega\subset\mathbb{R}^2$ be a bounded domain.
We say that $u$ is a \textit{Leray-Hopf solution} to the Navier-Stokes equations \eqref{eq:NSE}-\eqref{eq:noslip} on $\Omega\times (0,\infty)$
  if 
\begin{itemize}
\item $u\in C_{w}([0,\infty); L^{2}_{\sigma}(\Omega))\cap L^{2}(0,\infty; W^{1,2}_{0,\sigma}(\Omega))$, 
\item $u$ satisfies 
 \begin{equation}\label{eq:2dNSweakform}
 \begin{split}
 &\int\limits_{0}^{\infty}\int\limits_{\Omega} u\cdot\partial_{t}\theta + \nu u\cdot\Delta\theta + u\otimes u:\nabla \theta\, dxdt=0\\
 &\textrm{for all divergence-free}\quad  \theta\in C^{\infty}_{0}(\Omega\times (0,\infty)),
 \end{split}
 \end{equation}
\item $u$ satisfies the global energy inequality
\begin{equation}\label{energyinequalityLeray}
\|u(\cdot,t)\|_{L^{2}(\Omega)}^2+2\nu\int\limits_{0}^{t}\int\limits_{\Omega}|\nabla u|^2 dyds\leq \|u(\cdot,0)\|_{L^{2}(\Omega)}^2\,\,\,\,\textrm{for}\,\,\textrm{all}\,\,t\in [0,\infty).
\end{equation}
\end{itemize}
\end{definition}
Let us recall the standard definition of a weak solution to the 2D Euler equations.
 \begin{definition}\label{def:weaksol}[weak solution of 2D Euler]
 
 Let $\Omega\subset\mathbb{R}^2$ be a bounded domain. We say 
$u^{E}\in L^{2}(0,T;L^{2}(\Omega))$ is a weak solution of the 2D Euler equation in $\Omega\times (0,T)$ if
 \begin{equation}\label{eq:2deulerweakform}
 \int\limits_{0}^{T}\int\limits_{\Omega} u^{E}\cdot\partial_{t}\theta + u^{E}\otimes u^{E}:\nabla \theta\, dxdt=0\quad  \textrm{for all divergence-free}\quad  \theta\in C^{\infty}_{0}(\Omega\times (0,T)).
 \end{equation}
 \end{definition}
\section{Absence of local anomalous dissipation}
\subsection{A priori estimates}
\begin{proposition}\label{prop:apriori1}
Let $\nu>0$, $1<p<\infty$, $R_0>0$ and $T>0$.
Let 
\begin{equation}\label{eq:initialvort1}
 \omega_{0}^{\nu}\in L^{p}(B(R_0)).   
\end{equation}
Suppose $u^{\nu}:B(R_0)\times (0,\infty)\rightarrow \mathbb{R}^2$ is smooth and divergence-free and that $\omega^{\nu}:B(R_0)\times (0,\infty)\rightarrow\mathbb{R}$ is a smooth function satisfying
\begin{equation}\label{eq:vort1}
\partial_{t}\omega^{\nu}-\nu\Delta \omega^{\nu}+u^{\nu}\cdot\nabla \omega^{\nu}=0,\quad\omega^{\nu}(x,0)=\omega_{0}^{\nu}(x)\quad\textrm{in}\quad B(R_0)\times (0,\infty),
\end{equation}
with 
\begin{equation}\label{eq:vortcontinuityLp}
\lim_{t\rightarrow 0^+}\|\omega^{\nu}(\cdot,t)-\omega_{0}^{\nu}\|_{L^{p}(B(R_0))}=0
\end{equation}
and
\begin{equation}\label{eq:bddtotalspeedsmall}
\|u^{\nu}\|_{L^{1}(0,T; L^{\infty}(B(R_0)))}<\frac{R_0}{8}.
\end{equation}
Then, under the above assumptions the following apriori estimate holds:
\begin{equation}\label{eq:aprioriest1anyp}
\begin{split}
&\sup_{t\in (0,T)}\int\limits_{B(\frac{3R_0}{8})} |\omega^{\nu}(x,t)|^p dx+\nu\int\limits_{0}^{T}\int\limits_{B(\frac{3R_0}{8})} |\nabla|\omega^{\nu}|^{\frac{p}{2}}|^2 dxds\\
&\leq C(p)\Big(\int\limits_{B(R_0)} |\omega_{0}^{\nu}(x)|^p dx+\frac{\nu}{R_0^2}\int\limits_{0}^{T} \int\limits_{B(R_0)} |\omega^\nu|^p dxds\Big).
\end{split}
\end{equation}
Moreover, when $1<p\leq 2$ the following estimate holds:
\begin{equation}\label{eq:aprioriest1}
\begin{split}
&\sup_{t\in (0,T)}\int\limits_{B(\frac{3R_0}{8})} |\omega^{\nu}(x,t)|^p dx+\nu\int\limits_{0}^{T}\int\limits_{B(\frac{3R_0}{8})} |\nabla|\omega^{\nu}|^{\frac{p}{2}}|^2 dxds\\
&\leq C(p)\Big(\int\limits_{B(R_0)} |\omega_{0}^{\nu}(x)|^p dx+\frac{(\nu T)^{1-\frac{p}{2}}}{R_0^{p}}\Big(\nu\int\limits_{0}^{T} \int\limits_{B(R_0)} |\omega^\nu|^2 dxds\Big)^{\frac{p}{2}}\Big).
\end{split}
\end{equation}
\end{proposition}
\begin{proof}
In proving the above Proposition, we follow arguments from \cite{AD23} and \cite{CW23}, which are in the context of local Cacciopolli estimates for parabolic equations with divergence-free drifts and localized regularity criteria for the 3D Euler equations respectively.
Following \cite{AD23} and \cite{CW23}, we define the modulation parameter $\lambda:[0,T]\rightarrow [0,\infty)$ by
\begin{equation}\label{eq:modulation}
\dot{\lambda}(t)=-\frac{2\|u^{\nu}(\cdot,t)\|_{L^{\infty}(B(R_0))}}{R_0},\qquad \lambda(0)=1.
\end{equation} 
From the assumption \eqref{eq:bddtotalspeedsmall} we see that
\begin{equation}\label{eq:modulationbounds}
\frac{3}{4}\leq\lambda(t)\leq 1\quad\forall t\in [0,T].
\end{equation}
Next define the functions $\tilde{\omega}^{\nu}:B(R_0)\times (0,T)\rightarrow\mathbb{R}$ and $\tilde{u}^{\nu}:B(R_0)\times (0,T)\rightarrow\mathbb{R}^2$ by
\begin{equation}\label{eq:tildeuomegadef}
\tilde{\omega}^{\nu}(y,t):= \omega^{\nu}(\lambda(t)y, t)\quad\textrm{and}\quad \tilde{u}^{\nu}(y,t)=u^{\nu}(\lambda(t)y,t)-\dot{\lambda}(t)y.
\end{equation}
Using \eqref{eq:vort1} gives that $\tilde{\omega}^{\nu}$ satisfies
\begin{equation}\label{eq:tildevorteq2}
\partial_{t}\tilde{\omega}^{\nu}-\frac{\nu}{\lambda(t)^2}\Delta_{y} \tilde{\omega}^{\nu}+\frac{\tilde{u}^{\nu}}{\lambda(t)}\cdot\nabla_{y} \tilde{\omega}^{\nu}=0,\quad\tilde{\omega}^{\nu}(y,0)=\omega_{0}^{\nu}(y)\quad\textrm{in}\quad B(R_0)\times (0,T).
\end{equation}
From the fact that $u^{\nu}$ is divergence-free, we see from \eqref{eq:tildeuomegadef} that in $B(R_0)\times (0,T)$.
\begin{equation}\label{eq:utildenudiv}
\nabla\cdot\tilde{u}^{\nu}(y,s)=-2\dot{\lambda}(s)=\frac{4\|u^{\nu}(\cdot,s)\|_{L^{\infty}(B(R_0))}}{R_0}.
\end{equation}
Arguing as in \cite{AD23} we have
\begin{equation}\label{eq:utilderadial}
\begin{split}
&\tilde{u}^{\nu}(y,s)\cdot\frac{y}{|y|}=u^{\nu}(\lambda(t)y,t)\cdot\frac{y}{|y|}+\frac{2|y|\|u^{\nu}(\cdot,t)\|_{L^{\infty}(B(R_0)}}{R_0}\geq 0\\
&\textrm{for}\quad (y,s)\in B(R_0)\setminus B(\tfrac{R_0}{2})\times (0,T).
\end{split}
\end{equation}
Let $\phi:\mathbb{R}\rightarrow [0,1]$ be a smooth even function that is decreasing on $[0,\infty)$  with
\begin{itemize}
\item $\phi(r)=1$ on $[-\frac{R_0}{2},\frac{R_0}{2}]$,
\item $\supp\phi\subset [-R_0, R_0]$ and
\item $|\nabla^{k}\phi(r)|\lesssim_{k} R_0^{-k}\quad\forall k\in\mathbb{N}.$
\end{itemize}
Define the cut-off function $\varphi:\mathbb{R}^2\rightarrow \mathbb{R}$
\begin{equation}\label{eq:cutoffdefprop1}
\varphi(y):=\phi(|y|).
\end{equation}
Then from \eqref{eq:utilderadial} and the decreasing property of $\phi$, we have that
\begin{equation}\label{eq:weakdivergencesign}
\begin{split}
&\tilde{u}^{\nu}(y,s)\cdot\nabla\varphi(y)=\tilde{u}^{\nu}(y,s)\cdot\frac{y}{|y|}\phi'(|y|)\leq 0\\
&\textrm{for}\quad (y,s)\in B(R_0)\setminus B(\tfrac{R_0}{2})\times (0,T).
\end{split}
\end{equation}
Multiplying \eqref{eq:tildevorteq2} by $|\tilde{\omega}^{\nu}|^{p-2}\tilde{\omega}^{\nu}\varphi$, integrating and integrating by parts gives that for $t\in (0,T)$
\begin{equation}\label{eq:localLpvort}
\begin{split}
&\int\limits_{B(R_0)} |\tilde{\omega}^{\nu}(y,t)|^p\varphi(y) dy+\frac{4\nu(p-1)}{p}\int\limits_{0}^{t}\int\limits_{B(R_0)}\frac{\varphi(y)|\nabla|\tilde{\omega}^{\nu}|^{\frac{p}{2}}|^2 (y,s)}{(\lambda(s))^2} dyds\\
&=\int\limits_{0}^{t}\int\limits_{B(R_0)}\frac{1}{\lambda(s)}\Big(\nabla\cdot \tilde{u}^{\nu}(y,s)|\tilde{\omega}^{\nu} (y,s)|^p\varphi(y)+\tilde{u}^{\nu}(y,s)\cdot\nabla\varphi(y) |\tilde{\omega}^{\nu} (y,s)|^p\Big) dyds\\
&+\int\limits_{B(R_0)} |{\omega_0^{\nu}}(y)|^p\varphi(y)dy+\nu\int\limits_{0}^{t}\int\limits_{B(R_0)}\frac{\Delta\varphi(y)|\tilde{\omega}^{\nu}(y,s)|^{p}}{(\lambda(s))^2} dyds.
\end{split}
\end{equation}
Using this, \eqref{eq:utildenudiv} and \eqref{eq:weakdivergencesign}, as well as using \eqref{eq:modulationbounds} and properties of the cut-off function, yields that for $t\in (0,T)$
\begin{equation}\label{eq:localLpvortinequality}
\begin{split}
&\int\limits_{B(R_0)} |\tilde{\omega}^{\nu}(y,t)|^p\varphi(y) dy+\frac{4\nu(p-1)}{p}\int\limits_{0}^{t}\int\limits_{B(R_0)}\frac{\varphi(y)|\nabla|\tilde{\omega}^{\nu}|^{\frac{p}{2}}|^2 (y,s)}{(\lambda(s))^2} dyds\\
&\leq \int\limits_{0}^{t}\frac{d}{ds}\log\Big(\frac{1}{(\lambda(s))^2}\Big)\int\limits_{B(R_0)}|\tilde{\omega}^{\nu} (y,s)|^p\varphi(y)dyds+\int\limits_{B(R_0)} |{\omega_0^{\nu}}(y)|^pdy\\
&+\frac{C\nu}{R_0^2}\int\limits_{0}^{t}\int\limits_{B(R_0)}|{\omega}^{\nu}(y,s)|^{p} dyds.
\end{split}
\end{equation}
By applying Gronwall's lemma to this, we obtain that for $t\in (0,T)$
\begin{equation}\label{eq:omegatildegronwall}
\begin{split}
&\int\limits_{B(R_0)} |\tilde{\omega}^{\nu}(y,t)|^p\varphi(y) dy\leq \frac{1}{(\lambda(t))^2}\Big(\int\limits_{B(R_0)} |{\omega_0^{\nu}}(y)|^p dy+\\
&+\frac{C\nu}{R_0^2}\int\limits_{0}^{t}\int\limits_{B(R_0)}|{\omega}^{\nu}(y,s)|^{p} dyds\Big)\\
&\lesssim \int\limits_{B(R_0)} |{\omega_0^{\nu}}(y)|^p dy+\frac{\nu}{R_0^2}\int\limits_{0}^{t}\int\limits_{B(R_0)}|{\omega}^{\nu}(x,s)|^{p} dxds.
\end{split}
\end{equation}
Using  \eqref{eq:modulationbounds} and properties of the cut-off function,  also yields that for $t\in (0,T)$
\begin{equation}\label{eq:vortrelation}
\begin{split}
&\int\limits_{B(\frac{3R_0}{8})}|\omega^{\nu}(x,t)|^p dx+\nu \int\limits_{0}^{t}\int\limits_{B(\frac{3R_0}{8})} |\nabla|\omega^{\nu}|^{\frac{p}{2}}|^2 dxds\lesssim_{p}\\
&\int\limits_{B(R_0)} |\tilde{\omega}^{\nu}(y,t)|^p\varphi(y) dy+\frac{4\nu(p-1)}{p}\int\limits_{0}^{t}\int\limits_{B(R_0)}\frac{\varphi(y)|\nabla|\tilde{\omega}^{\nu}|^{\frac{p}{2}}|^2 (y,s)}{(\lambda(s))^2} dyds.
\end{split}
\end{equation}
Combining this with \eqref{eq:localLpvortinequality}-\eqref{eq:omegatildegronwall} gives 
\begin{equation}\label{eq:aprioriestvortLp}
\begin{split}
&\sup_{t\in (0,T)}\int\limits_{B(\frac{3R_0}{8})} |\omega^{\nu}(x,t)|^p dx+\nu\int\limits_{0}^{T}\int\limits_{B(\frac{3R_0}{8})} |\nabla|\omega^{\nu}|^{\frac{p}{2}}|^2 dxds\\
&\leq C(p)\Big(\int\limits_{B(R_0)} |\omega_{0}^{\nu}(x)|^p dx+\frac{\nu}{R_0^2}\int\limits_{0}^{T} \int\limits_{B(R_0)} |\omega^\nu|^p dxds\Big).
\end{split}
\end{equation}
We then obtain \eqref{eq:aprioriest1} by applying H\"{o}lder's inequality to the second integrand on the right-hand-side of the above inequality.
\end{proof}
Next, we show how the apriori estimates in Proposition \ref{prop:apriori1}, together with arguments from \cite{DP25dissconc}, imply that the dissipation can be estimated in terms of an $L^2$-based structure function away from the boundary. This will play a key role in the proof of Theorem \ref{thm:localenergybalancecompactL3}.
\begin{proposition}\label{prop:aprioriest1.1}
Suppose that $u^{\nu}:B(R_0)\times (0,\infty)\rightarrow \mathbb{R}^2$ is a smooth solution to the 2D Navier-Stokes equations on $B(R_0)\times (0,\infty)$. Moreover, suppose $u^{\nu}$ satisfies the assumption \eqref{eq:bddtotalspeedsmall} on $B(R_0)\times (0,T)$.
Then, the following apriori estimate holds for 
$0<\delta\leq T\quad\textrm{and}\quad 0<\nu< (\frac{5R_0}{8})^{2}: $
\begin{equation}\label{eq:aprioriest1.1}
\begin{split}
&\nu\int\limits_{\delta}^{T} \int\limits_{B(\frac{R_0}{8})} |\nabla u^{\nu}|^2 dxds\lesssim\Big(1+\frac{\nu}{R_0^2}+\frac{1}{\delta}+\frac{\nu T}{\delta R_0^2}\Big)^{\frac{1}{2}}\times\\
 &\Big(\nu\int\limits_{\frac{\delta}{2}}^{T} \int\limits_{B(R_0)} |\nabla u^{\nu}|^2 dxds\Big)^{\frac{1}{2}}\Big(\int\limits_{{\delta}}^{T}\dashint_{B(\nu^{\frac{1}{2}})}\int_{B(\frac{3 R_0}{8})}|u^{\nu}(x-y,s)-u^{\nu}(x,s)|^2 dxdyds\Big)^{\frac{1}{2}}.
\end{split}
\end{equation}
\end{proposition}
\begin{proof}
The proof combines the estimates of Proposition \ref{prop:apriori1} with arguments present in \cite[Proposition 3.2]{DP25dissconc} and \cite[Proposition 4.1]{DP25dissconc}. 
Below we will let $u^{\nu}_{\nu^{\frac{1}{2}}}$ be the mollification of $u^{\nu}$:
\begin{equation}\label{eq:umollifydef}
u^{\nu}_{\nu^{\frac{1}{2}}}(x,s):=\frac{1}{\nu}\int\limits_{B(\nu^{\frac{1}{2}})} \rho\Big(\frac{y}{\nu^{\frac{1}{2}}}\Big) u^{\nu}(x-y,s) dy\quad\textrm{for}\quad (x,s)\in B\Big(\frac{3R_0}{8}\Big)\times(0,T),
\end{equation}
where $\rho$ is a standard mollifier.
We will also use the following estimates that follow by classical estimates for mollified functions, see \cite[Section 2.1]{DP25dissconc} and \cite[(4.4)]{DP25dissconc}. Namely, for $s\in (0,T)$
\begin{equation}\label{eq:umollifierestL2}
\|u^{\nu}_{\nu^{\frac{1}{2}}}(\cdot,s)-u^{\nu}(\cdot,s)\|_{L^{2}(B(\frac{3R_0}{8}))}^2\lesssim \dashint_{B(\nu^{\frac{1}{2}})}\int_{B(\frac{3 R_0}{8})}|u^{\nu}(x-y,s)-u^{\nu}(x,s)|^2 dxdy,
\end{equation}
\begin{equation}\label{eq:umollifierestgradL2}
\nu\|\nabla u^{\nu}_{\nu^{\frac{1}{2}}}(\cdot,s)\|_{L^{2}(B(\frac{3R_0}{8}))}^2 \lesssim \dashint_{B(\nu^{\frac{1}{2}})}\int_{B(\frac{3 R_0}{8})}|u^{\nu}(x-y,s)-u^{\nu}(x,s)|^2 dxdy.
\end{equation}
Now for $\omega^{\nu}=\partial_{1}u^{\nu}_{2}-\partial_{2}u^{\nu}_{1}$, we proceed by estimating
$$\nu^2 \int\limits_{\delta}^{T}\int\limits_{B(\frac{3R_0}{8})} |\nabla \omega^{\nu}|^2 dxds. $$
Due to the assumption \eqref{eq:bddtotalspeedsmall}, identical arguments to Proposition \ref{prop:apriori1} yields that  for any $t\in (0,T)$
\begin{equation}\label{eq:vortenergyfromt}
\begin{split}
&\sup_{s\in (t,T)}\int\limits_{B(\frac{3R_0}{8})} |\omega^{\nu}(x,s)|^2 dx+\nu\int\limits_{t}^{T}\int\limits_{B(\frac{3R_0}{8})} |\nabla\omega^{\nu}|^2 dxds\\
&\lesssim\Big(\int\limits_{B(R_0)} |\omega^{\nu}(x,t)|^2 dx+\frac{\nu}{R_0^2}\int\limits_{t}^{T} \int\limits_{B(R_0)} |\omega^\nu|^2 dxds\Big).
\end{split}
\end{equation}
Multiplying this by $\nu$ and integrating over $t\in (\frac{\delta}{2},T)$ gives 
\begin{equation}\label{eq:vortenergyfromtintegrate}
\begin{split}
&\nu^2 \int\limits_{\frac{\delta}{2}}^{T}\int\limits_{t}^{T}\int\limits_{B(\frac{3R_0}{8})} |\nabla\omega^{\nu}|^2 dxdsdt\\
&\lesssim \nu\int\limits_{\frac{\delta}{2}}^{T}\int\limits_{B(R_0)} |\omega^{\nu}(x,t)|^2 dxdt+\frac{\nu^2}{R_0^2}\int\limits_{\frac{\delta}{2}}^{T}\int\limits_{t}^{T} \int\limits_{B(R_0)} |\omega^\nu|^2 dxdsdt.
\end{split}
\end{equation}
By Fubini, we get
\begin{equation}\label{eq:vortenergyfromtintegfubini}
\begin{split}
&\nu^2 \int\limits_{\frac{\delta}{2}}^{T} \int\limits_{B(\frac{3R_0}{8})} (s-\tfrac{\delta}{2})|\nabla\omega^{\nu}(x,s)|^2 dxds\\
&\lesssim \nu\int\limits_{\frac{\delta}{2}}^{T}\int\limits_{B(R_0)} |\omega^{\nu}|^2 dxdt+\frac{\nu^2}{R_0^2}\int\limits_{\frac{\delta}{2}}^{T} \int\limits_{B(R_0)} (s-\tfrac{\delta}{2})|\omega^{\nu}(x,s)|^2 dxds.
\end{split}
\end{equation}
Thus for $0<\delta\leq T$ we have
\begin{equation}\label{eq:gradvortestdelta}
\nu^2 \int\limits_{\delta}^{T} \int\limits_{B(\frac{3R_0}{8})} |\nabla\omega^{\nu}|^2 dxds
\lesssim \Big(\frac{1}{\delta}+\frac{\nu T}{\delta R_0^2}\Big)\nu\int\limits_{\frac{\delta}{2}}^{T}\int\limits_{B(R_0)} |\nabla u^{\nu}|^2 dxds.
\end{equation}
From this and \eqref{eq:umollifierestL2}-\eqref{eq:umollifierestgradL2}, we now show how to obtain \eqref{eq:aprioriest1.1} using arguments from \cite{DP25dissconc}.
Let $\varphi:\mathbb{R}^2\rightarrow [0,1]$ be a test function such that
\begin{itemize}
\item $\supp(\varphi)\subset B(\frac{3R_0}{8})$,
\item $\varphi=1$ on $B(\frac{R_0}{8})$ and 
\item $|\nabla^{k}\varphi(x)|\lesssim_{k} R_0^{-k}$ for all $x\in\mathbb{R}^2$ and $k=1,2\ldots$.
\end{itemize}
Then integrating by parts gives
\begin{equation}\label{eq:dissestIBP}
\begin{split}
&\nu\int\limits_{\delta}^{T} \int\limits_{B(\frac{R_0}{8})} |\nabla u^{\nu}|^2 dxds\leq \nu\int\limits_{\delta}^{T} \int\limits_{B(\frac{3R_0}{8})} |\nabla u^{\nu}|^2 \varphi dxds\\
&=\nu\int\limits_{\delta}^{T} \int\limits_{B(\frac{3R_0}{8})} \nabla u^{\nu}:\nabla (u^{\nu}-u^{\nu}_{\nu^{\frac{1}{2}}}) \varphi dxds+\nu\int\limits_{\delta}^{T} \int\limits_{B(\frac{3R_0}{8})} \nabla u^{\nu}:\nabla u^{\nu}_{\nu^{\frac{1}{2}}} \varphi dxds\\
&=-\nu\int\limits_{\delta}^{T} \int\limits_{B(\frac{3R_0}{8})} \Delta u^{\nu}\cdot (u^{\nu}-u^{\nu}_{\nu^{\frac{1}{2}}}) \varphi dxds-\nu\int\limits_{\delta}^{T} \int\limits_{B(\frac{3R_0}{8})} \nabla u^{\nu}: (u^{\nu}-u^{\nu}_{\nu^{\frac{1}{2}}})\otimes\nabla \varphi dxds\\
&+\nu\int\limits_{\delta}^{T} \int\limits_{B(\frac{3R_0}{8})} \nabla u^{\nu}:\nabla u^{\nu}_{\nu^{\frac{1}{2}}} \varphi dxds.
\end{split}
\end{equation}
Now $$-\Delta u^{\nu}(x,s)=\begin{bmatrix}
\partial_{2}\omega^{\nu}(x,s)\\
-\partial_{1}\omega^{\nu}(x,s)
\end{bmatrix}
.
$$
So from this, \eqref{eq:dissestIBP} and H\"{o}lder's inequality we get
\begin{equation}\label{eq:dissestIBPholder}
\begin{split}
&\nu\int\limits_{\delta}^{T} \int\limits_{B(\frac{R_0}{8})} |\nabla u^{\nu}|^2 dxds\leq\\
& \Big(\nu^2\int\limits_{\delta}^{T} \int\limits_{B(\frac{3R_0}{8})} |\nabla \omega^{\nu}|^2 dxds\Big)^{\frac{1}{2}}\Big(\int\limits_{\delta}^{T} \int\limits_{B(\frac{3R_0}{8})} |u^{\nu}-u^{\nu}_{\nu^{\frac{1}{2}}}|^2 dxds\Big)^{\frac{1}{2}}\\
&+\frac{\nu^{\frac{1}{2}}}{R_0}\Big(\nu\int\limits_{\delta}^{T} \int\limits_{B(\frac{3R_0}{8})} |\nabla u^{\nu}|^2 dxds\Big)^{\frac{1}{2}}\Big(\int\limits_{\delta}^{T} \int\limits_{B(\frac{3R_0}{8})} |u^{\nu}-u^{\nu}_{\nu^{\frac{1}{2}}}|^2 dxds\Big)^{\frac{1}{2}}\\
&+\Big(\nu\int\limits_{\delta}^{T} \int\limits_{B(\frac{3R_0}{8})} |\nabla u^{\nu}|^2 dxds\Big)^{\frac{1}{2}}\Big(\nu\int\limits_{\delta}^{T} \int\limits_{B(\frac{3R_0}{8})} |\nabla u^{\nu}_{\nu^{\frac{1}{2}}}|^2 dxds\Big)^{\frac{1}{2}}.
\end{split}
\end{equation}
Substituting \eqref{eq:gradvortestdelta} and \eqref{eq:umollifierestL2}-\eqref{eq:umollifierestgradL2} into this, we obtain the desired conclusion.
\end{proof}
\begin{proposition}\label{eq:aprioriest2}
Suppose $u^{\nu}:B(R_0)\times (0,\infty)\rightarrow \mathbb{R}^2$ is smooth and divergence-free. The following apriori estimate holds for $T>0$ and $R_0>0$
\begin{equation}\label{eq:apriori2}
\begin{split}
&\nu\int\limits_{0}^{T}\int\limits_{B(\frac{R_0}{8})} |\nabla u^{\nu}|^2 dxds\lesssim_{p} R_0^{-2}\nu T \sup_{t\in (0,T)} \int\limits_{B(\frac{3R_0}{16})}|u^{\nu} (x,t)|^2 dx\\
&+\nu \int\limits_{0}^{T}\int\limits_{B(\frac{3R_0}{16})}|\omega^\nu|^2 dxds.
\end{split}
\end{equation}
Here, $\omega^{\nu}=\partial_{1} u^{\nu}_{2}-\partial_{2}u^{\nu}_{1}.$
\end{proposition}
\begin{proof}
 Localized Calder\'on-Zygmund estimates are known, see for example \cite[Proposition 2.9]{DP25dissconc}. We provide a proof for the reader's convenience.

Fix $s\in (0,T)$. Let $\varphi:\mathbb{R}^2\rightarrow [0,1]$ be a test function such that
\begin{itemize}
\item $\supp(\varphi)\subset B(\frac{3R_0}{16})$,
\item $\varphi=1$ on $B(\frac{R_0}{8})$ and 
\item $|\nabla^{k}\varphi(x)|\lesssim_{k} R_0^{-k}$ for all $x\in\mathbb{R}^2$ and $k=1,2\ldots$.
\end{itemize}
Now $$-\Delta u^{\nu}(x,s)=\begin{bmatrix}
\partial_{2}\omega^{\nu}(x,s)\\
-\partial_{1}\omega^{\nu}(x,s)
\end{bmatrix}
.
$$
Taking the dot product of this with $u^{\nu}\varphi^2$, integrating and  then integrating by parts yields
\begin{equation}\label{eq:IBPvort}
\begin{split}
&\int\limits_{\mathbb{R}^2}(\varphi(x))^2|\nabla u^{\nu}(x,s)|^2 dx=\frac{1}{2}\int\limits_{\mathbb{R}^2}\Delta((\varphi(x))^2)| u^{\nu}(x,s)|^2 dx\\
&+\int\limits_{\mathbb{R}^2}\omega^{\nu}(x,s)(\partial_{1}(u_2^{\nu}(x,s)\varphi(x)^2)-\partial_{2}(u_1^{\nu}(x,s)\varphi(x)^2)) dx.
\end{split}
\end{equation}
Thus,
\begin{equation*}
\begin{split}
&\int\limits_{\mathbb{R}^2}(\varphi(x))^2|\nabla u^{\nu}(x,s)|^2 dx\lesssim \int\limits_{\mathbb{R}^2}|\Delta((\varphi(x))^2)|| u^{\nu}(x,s)|^2 dx\\
&+\int\limits_{\mathbb{R}^2}|\omega^{\nu}(x,s)||\nabla u^{\nu}(x,s)|\varphi(x)^2+ |\omega^{\nu}(x,s)||u^{\nu}(x,s)||\nabla \varphi(x)| \varphi(x)dx.
\end{split}
\end{equation*}
This together with Young's inequality yields
\begin{equation*}
\begin{split}
&\int\limits_{\mathbb{R}^2}(\varphi(x))^2|\nabla u^{\nu}(x,s)|^2 dx\lesssim \int\limits_{\mathbb{R}^2}(|\Delta((\varphi(x))^2)|+|\nabla\varphi(x)|^2)| u^{\nu}(x,s)|^2 dx\\
&+\int\limits_{\mathbb{R}^2}|\omega^{\nu}(x,s)|^2(\varphi(x))^2 dx.
\end{split}
\end{equation*}
Using this and taking into account the previously mentioned properties of the test function yields that for all $s\in (0,T)$
\begin{equation}\label{eq:IBPvortyoungs}
\begin{split}
&\int\limits_{B(\frac{R_0}{8})}|\nabla u^{\nu}(x,s)|^2 dx\lesssim R_0^{-2}\int\limits_{B(\frac{3R_0}{16})}|u^{\nu}(x,s)|^2 dx+\int\limits_{B(\frac{3R_0}{16})}|\omega^{\nu}(x,s)|^2 dx.
\end{split}
\end{equation}
Multiplying \eqref{eq:IBPvortyoungs} by $\nu$, integrating over $(0,T)$ and using H\"{o}lder's inequality on the first term then gives the desired conclusion.
\end{proof}
\begin{proposition}\label{eq:aprioriest3}
Suppose that the same assumptions as in Proposition \ref{prop:apriori1} hold with $1<p\leq 2$. Furthermore, suppose that $\omega^{\nu}=\partial_{1} u^{\nu}_{2}-\partial_{2}u^{\nu}_{1}.$
Then, the following apriori estimate holds
\begin{equation}\label{eq:apriori3}
\begin{split}
&\nu\int\limits_{0}^{T}\int\limits_{B(\frac{R_0}{8})} |\nabla u^{\nu}|^2 dxds\lesssim_{p} R_0^{-2}\nu T \sup_{t\in (0,T)} \int\limits_{B(\frac{3R_0}{16})}|u^{\nu} (x,t)|^2 dx+\\
&\max(\nu TR_0^{-\frac{2(2-p)}{p}},(\nu T)^{2-\frac{2}{p}})\Big(\int\limits_{B(R_0)} |\omega_{0}^{\nu}(x)|^p dx+\\
&\frac{(\nu T)^{1-\frac{p}{2}}}{R_0^{p}}\Big(\nu\int\limits_{0}^{T} \int\limits_{B(R_0)} |\nabla u^\nu|^2 dxds\Big)^{\frac{p}{2}}\Big)^{\frac{2}{p}}.
\end{split}
\end{equation}
\end{proposition}
\begin{proof}
Fix $s\in (0,T)$. Let $\varphi:\mathbb{R}^2\rightarrow [0,1]$ be a test function such that
\begin{itemize}
\item $\supp(\varphi)\subset B(\frac{3R_0}{8})$,
\item $\varphi=1$ on $B(\frac{3R_0}{16})$ and 
\item $|\nabla^{k}\varphi(x)|\lesssim_{k} R_0^{-k}$ for all $x\in\mathbb{R}^2$ and $k=1,2\ldots$.
\end{itemize}
Following \cite{CFLS16}, we apply the Gagliardo-Nirenberg inequality (see, for example, \cite[Theorem 7.1]{MRR13}) to $\varphi |\omega^{\nu}|^{\frac{p}{2}}$ yielding
$$\|\varphi |\omega^{\nu}(\cdot,s)|^{\frac{p}{2}}\|_{L^{\frac{4}{p}}(B(\frac{3R_0}{8}))}\lesssim_{p} \|\varphi |\omega^{\nu}(\cdot,s)|^{\frac{p}{2}}\|_{L^{2}(B(\frac{3R_0}{8}))}^{\frac{p}{2}} \|\nabla(\varphi |\omega^{\nu}(\cdot,s)|^{\frac{p}{2}})\|_{L^{2}(B(\frac{3R_0}{8}))}^{\frac{2-p}{2}}.$$
Thus,
\begin{equation}\label{eq:GNvorttimeslice}
\begin{split}
\|\omega^{\nu}(\cdot,s)\|_{L^{2}(B(\frac{3R_0}{16}))}^{2}\lesssim_{p}& R_0^{-\frac{2(2-p)}{p}}\| |\omega^{\nu}(\cdot,s)|^{\frac{p}{2}}\|_{L^{2}(B(\frac{3 R_0}{8}))}^{\frac{4}{p}}+\\
&\| |\omega^{\nu}(\cdot,s)|^{\frac{p}{2}}\|_{L^{2}(B(\frac{3R_0}{8}))}^{2} \|\nabla |\omega^{\nu}(\cdot,s)|^{\frac{p}{2}}\|_{L^{2}(B(\frac{3R_0}{8}))}^{\frac{2(2-p)}{p}}.
\end{split}
\end{equation}
Integrating in time over $(0,T)$ and applying H\"{o}lder's inequality gives
\begin{equation}\label{eq:GNvortinteg}
\begin{split}
&\nu\int\limits_{0}^{T}\|\omega^{\nu}(\cdot,s)\|_{L^{2}(B(\frac{3R_0}{16}))}^{2} ds \lesssim_{p} \nu TR_0^{-\frac{2(2-p)}{p}}\sup_{s\in (0,T)}\| |\omega^{\nu}(\cdot,s)|^{\frac{p}{2}}\|_{L^{2}(B(\frac{3 R_0}{8}))}^{\frac{4}{p}}+\\
&\sup_{s\in (0,T)}\| |\omega^{\nu}(\cdot,s)|^{\frac{p}{2}}\|_{L^{2}(B(\frac{3R_0}{8}))}^{2} \nu\int\limits_{0}^{T}\|\nabla |\omega^{\nu}(\cdot,s)|^{\frac{p}{2}}\|_{L^{2}(B(\frac{3R_0}{8}))}^{\frac{2(2-p)}{p}}ds\\
&\lesssim_{p}\nu TR_0^{-\frac{2(2-p)}{p}}\sup_{s\in (0,T)}\| |\omega^{\nu}(\cdot,s)|^{\frac{p}{2}}\|_{L^{2}(B(\frac{3 R_0}{8}))}^{\frac{4}{p}}+\\
&(\nu T)^{2-\frac{2}{p}}\sup_{s\in (0,T)}\| |\omega^{\nu}(\cdot,s)|^{\frac{p}{2}}\|_{L^{2}(B(\frac{3R_0}{8}))}^{2} \Big(\nu\int\limits_{0}^{T}\|\nabla |\omega^{\nu}(\cdot,s)|^{\frac{p}{2}}\|_{L^{2}(B(\frac{3R_0}{8}))}^{2}ds\Big)^{\frac{2-p}{p}}.
\end{split}
\end{equation}
Substituting in \eqref{eq:aprioriest1} from Proposition \ref{prop:apriori1}, together with the fact that $|\omega^{\nu}|\lesssim |\nabla u^{\nu}|$, gives 
\begin{equation}\label{eq:vortLpinitial}
\begin{split}
&\nu\int\limits_{0}^{T}\|\omega^{\nu}(\cdot,s)\|_{L^{2}(B(\frac{3R_0}{16}))}^{2} ds \lesssim_{p}
 \max(\nu TR_0^{-\frac{2(2-p)}{p}},(\nu T)^{2-\frac{2}{p}})\Big(\int\limits_{B(R_0)} |\omega_{0}^{\nu}(x)|^p dx+\\
&\frac{(\nu T)^{1-\frac{p}{2}}}{R_0^{p}}\Big(\nu\int\limits_{0}^{T} \int\limits_{B(R_0)} |\nabla u^\nu|^2 dxds\Big)^{\frac{p}{2}}\Big)^{\frac{2}{p}}.
\end{split}
\end{equation}
Proposition \ref{eq:aprioriest2} then gives 
\begin{equation} 
\label{eq:locdissipationconc}
\begin{split}
&\nu\int\limits_{0}^{T}\int\limits_{B(\frac{R_0}{8})} |\nabla u^{\nu}|^2 dxds\lesssim_{p} R_0^{-2}\nu T \sup_{t\in (0,T)} \int\limits_{B(\frac{3R_0}{16})}|u^{\nu} (x,t)|^2 dx\\
&+\nu \int\limits_{0}^{T}\int\limits_{B(\frac{3R_0}{16})}|\omega^\nu|^2 dxds\lesssim_{p}R_0^{-2}\nu T \sup_{t\in (0,T)} \int\limits_{B(\frac{3R_0}{16})}|u^{\nu} (x,t)|^2 dx+\\&
\max(\nu T R_0^{-\frac{2(2-p)}{p}},(\nu T)^{2-\frac{2}{p}})\Big(\int\limits_{B(R_0)} |\omega_{0}^{\nu}(x)|^p dx+\\
&\frac{(\nu T)^{1-\frac{p}{2}}}{R_0^{p}}\Big(\nu\int\limits_{0}^{T} \int\limits_{B(R_0)} |\nabla u^\nu|^2 dxds\Big)^{\frac{p}{2}}\Big)^{\frac{2}{p}}
\end{split}
\end{equation}
 as required.
\end{proof}
\subsection{Proof of Theorem \ref{thm:nolocaanondis}}
Fix $K\subset\Omega$ to be a compact subset and define
\begin{equation}\label{eq:R0def}
R_0:=\frac{\textrm{dist}\,(K,\partial\Omega)}{2}.
\end{equation}
From the assumption \ref{eq:L1plusinfinity} and H\"{o}lder's inequality, we have for any $z\in K$ and $0<S\leq T$ that for all $\nu>0$
\begin{equation}\label{eq:L1Linfinitybound1}
\|u^{\nu}\|_{L^{1}(0,S; L^{\infty}(B(z,R_0)))}\leq S^{\frac{\varepsilon}{1+\varepsilon}}\|u^{\nu}\|_{L^{1+\varepsilon}(0,S; L^{\infty}(\Omega(R_0)))}\leq  S^{\frac{\varepsilon}{1+\varepsilon}} N(R_0).
\end{equation}
Define $t^{*}_{K}=t^{*}_{K}(K, \Omega, T,\varepsilon, N)$ by
\begin{equation}\label{eq:tkdef}
t_{K}^{*}:=\min\Big(T,\frac{1}{2}\Big(\frac{R_0}{8N(R_0)}\Big)^{\frac{1+\varepsilon}{\varepsilon}}\Big).
\end{equation}
Then from \eqref{eq:L1Linfinitybound1} we see that for all $\nu>0$
\begin{equation}\label{eq:L1LinfinityR08}
\|u^{\nu}\|_{L^{1}(0,t_{K}^{*}; L^{\infty}(B(z,R_0)))}<\frac{R_0}{8}.
\end{equation}
In what follows, we will also use that $u^{\nu}$ satisfies the energy equality
\begin{equation}\label{eq:energyequality}
\|u^{\nu}(\cdot,t)\|_{L^{2}(\Omega)}^2+2\nu\int\limits_{0}^{t}\int\limits_{\Omega}|\nabla u^{\nu}(x,s)|^2 dxds=\|u_{0}^{\nu}\|_{L^{2}(\Omega)}^2\quad\forall t\geq 0.
\end{equation}
\textbf{Case 1: $1<p\leq 2$}\\
From \eqref{eq:L1LinfinityR08}, we can apply Proposition \ref{prop:apriori1}. This, together with the energy equality, \ref{eq:L2iduniform} and \ref{eq:vortLpuniform}, gives for any $z\in K$ 
\begin{equation}\label{eq:apriori1applied}
\begin{split}
&\sup_{t\in (0,t_{K}^{*})}\int\limits_{B(z,\frac{3R_0}{8})} |\omega^{\nu}(x,t)|^p dx \lesssim_{p} \Big(\int\limits_{B(z,R_0)} |\omega_{0}^{\nu}(x)|^p dx+\\
&\frac{(\nu t_{K}^{*})^{1-\frac{p}{2}}}{R_0^{p}}\Big(\nu\int\limits_{0}^{t_{K}^{*}} \int\limits_{B(z,R_0)} |\omega^\nu|^2 dxds\Big)^{\frac{p}{2}}\Big)\lesssim_{p} M(R_0)^p+(\nu t_{K}^{*})^{1-\frac{p}{2}}R_0^{-p}J^p.
\end{split}
\end{equation}
From \eqref{eq:L1LinfinityR08}, we can also apply Proposition \ref{eq:aprioriest3}. This, together with the energy equality and \ref{eq:L2iduniform}, \ref{eq:vortLpuniform}, gives 
\begin{equation}\label{eq:apriori3applied}
\begin{split}
&\nu\int\limits_{0}^{t_{K}^{*}}\int\limits_{B(z,\frac{R_0}{8})} |\nabla u^{\nu}|^2 dxds\lesssim_{p} R_0^{-2}\nu t_{K}^{*} \sup_{t\in (0,t_{K}^{*})} \int\limits_{B(z,\frac{3R_0}{16})}|u^{\nu} (x,t)|^2 dx+\\
&\max(\nu t_{K}^{*}R_0^{-\frac{2(2-p)}{p}},(\nu t_{K}^{*})^{2-\frac{2}{p}})\Big(\int\limits_{B(z,R_0)} |\omega_{0}^{\nu}(x)|^p dx+\\
&\frac{(\nu t_{K}^{*})^{1-\frac{p}{2}}}{R_0^{p}}\Big(\nu\int\limits_{0}^{t_{K}^{*}} \int\limits_{B(z,R_0)} |\nabla u^\nu|^2 dxds\Big)^{\frac{p}{2}}\Big)^{\frac{2}{p}}\lesssim_{p}R_0^{-2}\nu t_{K}^{*} J^2+\\
&\max(\nu t_{K}^{*}R_0^{-\frac{2(2-p)}{p}},(\nu t_{K}^{*})^{2-\frac{2}{p}})\Big((M(R_0))^p+\frac{(\nu t_{K}^{*})^{1-\frac{p}{2}}J^p}{R_0^{p}}\Big)^{\frac{2}{p}}.
\end{split}
\end{equation}
Since $K\subset \Omega$ is compact, we let $(z_{i})_{1\leq i\leq Q}$ be in $K$ such that we have the finite covering
$$K\subset\cup_{i=1}^{Q} B(z_{i}, \tfrac{R_0}{8}).$$
Applying \eqref{eq:apriori1applied} to each $z_{i}$ and summing over $i=1\ldots Q$ then gives that for each $t\in (0,t_{K}^{*})$ that
\begin{equation}\label{eq:vortLpoverKest}
\begin{split}
&\int\limits_{K} |\omega^{\nu}(x,t)|^p dx\leq \sum_{i=1}^{Q}\int\limits_{B(z_{i},\frac{R_0}{8})} |\omega^{\nu}(x,t)|^p dx\lesssim_{p} Q M(R_0)^p+Q(\nu t_{K}^{*})^{1-\frac{p}{2}}R_0^{-p}J^p.
\end{split}
\end{equation} 
Hence,
$$\sup_{\nu>0}\|\omega^{\nu}\|_{L^{\infty}(0,t_{K}^{*}; L^{p}(K))}<\infty $$ as required.
Applying \eqref{eq:apriori3applied} to each $z_{i}$ and summing over $i=1\ldots Q$ then gives
\begin{equation}\label{eq:dissipationoverKest}
\begin{split}
&\nu\int\limits_{0}^{t_{K}^{*}}\int\limits_{K} |\nabla u^{\nu}|^2 dxds\leq \sum_{i=1}^{Q}\nu\int\limits_{0}^{t_{K}^{*}}\int\limits_{B(z_{i},\frac{R_0}{8})} |\nabla u^{\nu}|^2 dxds\\
&\lesssim_{p} QR_0^{-2}\nu t_{K}^{*} J^2+
Q\max(\nu t_{K}^{*}R_0^{-\frac{2(2-p)}{p}},(\nu t_{K}^{*})^{2-\frac{2}{p}})\times\\
&\Big((M(R_0))^p+\frac{(\nu t_{K}^{*})^{1-\frac{p}{2}}J^p}{R_0^{p}}\Big)^{\frac{2}{p}}.
\end{split}
\end{equation} 
From this, we see that the leading order behavior as $\nu\rightarrow0$ is $\nu^{2-\frac{2}{p}}$. Hence,
$$ \lim_{v\rightarrow 0}\nu\int\limits_{0}^{t_{K}^{*}}\int\limits_{K} |\nabla u^{\nu}|^2 dxds=0$$
as required.\\
\textbf{Case 2: $2<p<\infty$}\\
 Assumption \ref{eq:vortLpuniform}, H\"{o}lder's inequality and case 1 immediately give \eqref{eq:thm1conclusion}. We therefore focus on showing the local vorticity bound \eqref{eq:thm1conclusionvortbound}.

Let us first focus on the case $2<p\leq 4$, before explaining how to treat general $2<p<\infty$.

Let $R_0$ be given by \eqref{eq:R0def} and let $t_{K}^{*}$ be given by \eqref{eq:tkdef} such that \eqref{eq:L1LinfinityR08} holds true. Define
\begin{equation}\label{eq:R0'def}
R_0':=\frac{3}{16}R_0=\frac{3\textrm{dist}\,(K,\partial\Omega)}{32}
\end{equation}
and define $t'_{K}=t'_{K}(K, \Omega, T,\varepsilon, N)$ by
\begin{equation}\label{eq:t'def}
t'_{K}:=\min\Big(T,\frac{1}{2}\Big(\frac{R'_0}{8N(R_0)}\Big)^{\frac{1+\varepsilon}{\varepsilon}},\Big).
\end{equation}
By \eqref{eq:L1LinfinityR08} we have
\begin{equation}\label{eq:L1Linfinityt'R0}
\|u^{\nu}\|_{L^{1}(0,t'_{K}; L^{\infty}(B(z,R_0)))}<\frac{R_0}{8}.
\end{equation}
From \eqref{eq:L1Linfinitybound1} and \eqref{eq:R0'def}-\eqref{eq:t'def}, we have that for any $z\in K$ 
\begin{equation}\label{eq:L1Linfinitybound1t'R0'}
\begin{split}
\|u^{\nu}\|_{L^{1}(0,t'_{K}; L^{\infty}(B(z,R'_0)))}&\leq \|u^{\nu}\|_{L^{1}(0,t'_{K}; L^{\infty}(B(z,R_0)))}\\
& \leq (t'_{K})^{\frac{\varepsilon}{1+\varepsilon}}\|u^{\nu}\|_{L^{1+\varepsilon}(0,t'_{K}; L^{\infty}(\Omega(R_0)))}\\
&\leq  (t'_{K})^{\frac{\varepsilon}{1+\varepsilon}} N(R_0)<\frac{R'_{0}}{8}.
\end{split}
\end{equation}
Now \eqref{eq:L1Linfinityt'R0} allows us to apply Proposition \ref{prop:apriori1}. 
This, together with the energy equality, \ref{eq:L2iduniform}, \ref{eq:vortLpuniform} and H\"{o}lder's inequality, implies that for $z\in K$:
\begin{equation}\label{eq:aprioriappliedp=2}
\begin{split}
&\sup_{t\in (0,t'_{K})}\int\limits_{B(z,2R'_0)} |\omega^{\nu}(x,t)|^2 dx+\nu\int\limits_{0}^{t'_{K}}\int\limits_{B(z,2R'_0)} |\nabla|\omega^{\nu}||^2 dxds\\
&\leq C\Big(\int\limits_{B(z,R_0)} |\omega_{0}^{\nu}(x)|^2 dx+\frac{\nu}{R_0^2}\int\limits_{0}^{t'_{K}} \int\limits_{B(z,R_0)} |\omega^\nu|^2 dxds\Big)\\
&\leq C(p) M(R_0)^{\frac{2}{p}}R_0^{2-\frac{4}{p}}+C R_0^{-2} J^2.
\end{split}
\end{equation}
Fix $s\in (0,t'_{K})$. Let $\varphi:\mathbb{R}^2\rightarrow [0,1]$ be a test function such that
\begin{itemize}
\item $\supp(\varphi)\subset B(z,2R'_{0})$,
\item $\varphi=1$ on $B(z,{R'_0})$ and 
\item $|\nabla^{k}\varphi(x)|\lesssim_{k} (R'_0)^{-k}$ for all $x\in\mathbb{R}^2$ and $k=1,2\ldots$.
\end{itemize}
Now applying the Gagliardo-Nirenberg inequality\footnote{This special case of the Gagliardo-Nirenberg is also referred to as `Ladyzhenskaya's inequality'.} to $\varphi |\omega^{\nu}|$ gives 
$$\nu\|\varphi |\omega^{\nu}(\cdot,s)|\|_{L^{4}(B(z,2R_0'))}^4\lesssim \|\varphi |\omega^{\nu}(\cdot,s)|\|_{L^{2}(B(z,2R_0'))}^{2} \nu\|\nabla(\varphi |\omega^{\nu}(\cdot,s)|)\|_{L^{2}(B(z,2R_0'))}^{2}.$$
Integrating this over $ (0,t'_{K})$ and using the properties of the test function $\varphi$ gives
\begin{equation}\label{eq:L4bootstrap}
\begin{split}
&\nu\int\limits_{0}^{t'_{K}}\int\limits_{B(z,R'_{0})}|\omega^{\nu}(x,s)|^4 dxds\lesssim \nu(R'_0)^{-1}\int\limits_{0}^{t'_{K}} \|\omega^{\nu}(\cdot,s)\|_{L^{2}(B(z,2R'_0))}^4 ds+\\
&\sup_{s\in (0,t'_{K})}\|\omega^{\nu}(\cdot,s)\|_{L^{2}(B(z,2R'_0))}^2 \Big(\nu \int\limits_{0}^{t'_{K}}\int\limits_{B(z,2R'_0)}|\nabla|\omega^{\nu}||^2 dxds\Big)\lesssim\\
& \nu t'_{K}(R'_0)^{-1}\sup_{s\in (0,t'_{K})} \|\omega^{\nu}(\cdot,s)\|_{L^{2}(B(z,2R'_0))}^4 +\\
&\sup_{s\in (0,t'_{K})}\|\omega^{\nu}(\cdot,s)\|_{L^{2}(B(z,2R'_0))}^2 \Big(\nu \int\limits_{0}^{t'_{K}}\int\limits_{B(z,2R'_0)}|\nabla|\omega^{\nu}||^2 dxds\Big).
\end{split} 
\end{equation}
Substituting in the bound \eqref{eq:aprioriappliedp=2} then gives
\begin{equation}\label{eq:L4bootstrapuniformfinal}
\begin{split}
&\nu\int\limits_{0}^{t'_{K}}\int\limits_{B(z,R'_{0})}|\omega^{\nu}|^4 dxds
\lesssim_{p} \max(1, \nu t'_{K}(R'_0)^{-1})(M(R_0)^{\frac{2}{p}}R_0^{2-\frac{4}{p}}+R_0^{-2} J^2)^2.
\end{split} 
\end{equation}
So by H\"{o}lder's inequality we have
\begin{equation}\label{eq:Lpbootstrapuniform}
\begin{split}
&\nu\int\limits_{0}^{t'_{K}}\int\limits_{B(z,R'_{0})}|\omega^{\nu}|^p dxds\lesssim_{p} (\nu t'_{K} (R'_{0})^2)^{1-\frac{4}{p}}\Big(\nu\int\limits_{0}^{t'_{K}}\int\limits_{B(z,R'_{0})}|\omega^{\nu}|^4 dxds\Big)^{\frac{4}{p}} \\
&\lesssim_{p} (\nu t'_{K} (R'_{0})^2)^{1-\frac{4}{p}}(\max(1, \nu t'_{K}(R'_0)^{-1})(M(R_0)^{\frac{2}{p}}R_0^{2-\frac{4}{p}}+ R_0^{-2} J^2)^2)^{\frac{4}{p}}.
\end{split}
\end{equation}
Recalling \eqref{eq:L1Linfinitybound1t'R0'}, we can now apply Propososition \ref{prop:apriori1} again. This, together with  \eqref{eq:Lpbootstrapuniform}, gives
\begin{equation}\label{eq:aprioriappliedp>2}
\begin{split}
&\sup_{t\in (0,t'_{K})}\int\limits_{B(z,\frac{3R'_0}{8})} |\omega^{\nu}(x,t)|^p dx+\nu\int\limits_{0}^{t'_{K}}\int\limits_{B(z,\frac{3 R'_0}{8})} |\nabla(|\omega^{\nu}|^{\frac{p}{2}})|^2 dxds\\
&\leq C(p)\Big(\int\limits_{B(z,R'_0)} |\omega_{0}^{\nu}(x)|^p dx+\frac{\nu}{(R'_0)^2}\int\limits_{0}^{t'_{K}} \int\limits_{B(z,R'_0)} |\omega^\nu|^p dxds\Big)\\
&\lesssim_{p} M(R_0)^p+\\
&(R'_0)^{-2}(\nu t'_{K} (R'_{0})^2)^{1-\frac{4}{p}}(\max(1, \nu t'_{K}(R'_0)^{-1})(M(R_0)^{\frac{2}{p}}R_0^{2-\frac{4}{p}}+ R_0^{-2} J^2)^2)^{\frac{4}{p}}.
\end{split}
\end{equation}
A similar covering argument to \textbf{Case 1} then gives that
\begin{equation}\label{eq:vortLpoverK}
\begin{split}
&\sup_{t\in (0,t'_{K}), \nu>0}\int\limits_{K} |\omega^{\nu}(x,t)|^p dx\lesssim_{p} Q'(K)M(R_0)^p+\\
&Q'(K)(R'_0)^{-2}(\nu t'_{K} (R'_{0})^2)^{1-\frac{4}{p}}(\max(1, \nu t'_{K}(R'_0)^{-1})(M(R_0)^{\frac{2}{p}}R_0^{2-\frac{4}{p}}+ R_0^{-2} J^2)^2)^{\frac{4}{p}}.
\end{split}
\end{equation}
Now if $4<p\leq 8$ instead, we would then use \eqref{eq:L4bootstrapuniformfinal} and the bound \eqref{eq:aprioriappliedp>2} for integrability exponent 4, together with the aforementioned Gagliardo-Nirenberg inequality for $\varphi |\omega^\nu|^2$, to get a uniform (in $\nu$) bound on
$$ \nu\int\limits_{0}^{t'_{K}}\int\limits_{B(z,\frac{3 R'_{0}}{16})}|\omega^{\nu}|^8 dxds.$$ 
For appropriately chosen $R''_{0}$ and $t''_{K}=t''_{K}(K,\Omega, T,\varepsilon,N)$, one can then apply Proposition \ref{prop:apriori1} to bound (uniformily in $\nu$)
$$\sup_{t\in (0,t''_{K})}\int\limits_{B(z,\frac{3R''_0}{8})} |\omega^{\nu}(x,t)|^p dx+\nu\int\limits_{0}^{t''_{K}}\int\limits_{B(z,\frac{3 R''_0}{8})} |\nabla(|\omega^{\nu}|^{\frac{p}{2}})|^2 dxds $$ and then 
$$ \sup_{t\in (0,t''_{K}),\nu>0}\int\limits_{K} |\omega^{\nu}(x,t)|^p dx$$ by a covering argument. Continuing iteratively in this way, we can get the desired conclusion  \eqref{eq:thm1conclusionvortbound} for any $2<p<\infty$.
\begin{remark}\label{rmk:moseriterations}[Comparison with Moser iterations]
The iteration procedure is analogous to the application of Moser iterations in \cite{AD23} to show local boundedness of solutions to parabolic equations with certain divergence-free drifts in $L^{1}_{t}L^{\infty}_{x}$. 
\end{remark}
\section{Local energy balance for the large scale approximation}
\subsection{Definitions and derivations}
In what follows we denote 
\begin{equation}\label{eq:Lapace2Dinversegrapperp}
\Delta_{2D}^{-1}=\frac{1}{2\pi}\log(|x|)\star\quad\textrm{and}\quad \nabla^{\perp}=\begin{bmatrix} -\partial_{2}\\
\partial_{1} \end{bmatrix}.
\end{equation}
Let $\Omega\subset\mathbb{R}^2$ be a bounded domain and let $u:\Omega\rightarrow \mathbb{R}^2$ be a sufficiently integrable function. Let $\varphi\in C^{\infty}_{0}(\Omega; [0,1])$. Following \cite{K23}, define the 2D \textit{local Leray projection/large scale approximation} by 
\begin{equation}\label{eq:locallerayprojectiondef}
\mathbb{P}_{\varphi}(u):=\varphi u-h_{\varphi}(u),\quad\textrm{with}\quad h_{\varphi}(u):=\nabla \nabla\cdot\Delta_{2D}^{-1}(u\varphi)+\nabla^{\perp}\Delta_{2D}^{-1}(\nabla^{\perp}\varphi\cdot u).
\end{equation}
As noticed in \cite{K23}, by Calder\'{o}n-Zygmund estimates we have that for $1<q<\infty$ and $k=1,2\ldots$ that
\begin{equation}\label{eq:CZlocalleray}
\|\nabla^{k}\mathbb{P}_{\varphi}(u)\|_{L^{q}(\mathbb{R}^2)}\lesssim_{q,\varphi,k} \|u\|_{W^{k,q}(\supp(\varphi))}.
\end{equation}
Moreover, for $2 < q<\infty$ we have that
\begin{equation}\label{eq:CZlocalleray0}
\|\mathbb{P}_{\varphi}(u)\|_{L^{q}(\mathbb{R}^2)}\lesssim_{q,\varphi} \|u\|_{L^{q}(\supp(\varphi))}.
\end{equation}
The 2D vector calculus identities for a vector field $F$ and scalar $\phi$
\begin{equation}\label{eq:laplace2D}
\Delta F=\nabla^{\perp}(\nabla^{\perp}\cdot F)+\nabla(\nabla\cdot F)\quad\textrm{and}
\end{equation}
\begin{equation}\label{eq:testcurl}
\nabla^{\perp}\cdot(\phi F)=\phi\nabla^{\perp}\cdot F+\nabla^{\perp}\phi\cdot F
\end{equation}
 give that if $u$ is sufficiently smooth then $\mathbb{P}_{\varphi}(u) $ can be equivalently expressed as
\begin{equation}\label{eq:locallerayprojectcurl}
\mathbb{P}_{\varphi}(u)=\nabla^{\perp} \Delta_{2D}^{-1}(\varphi\nabla^{\perp}\cdot u).
\end{equation}
In particular,  for $u$ being sufficiently integrable, $\mathbb{P}_{\varphi} u$ is divergence-free in the sense of distributions by a standard approximation procedure. Moreover if $u$ is sufficiently integrable and divergence-free in the sense of distributions, we equivalently have that $h_{\varphi}(u)$ is given by
\begin{equation}\label{eq:hdivergencefree}
h_{\varphi}(u):=\nabla \Delta_{2D}^{-1}(u\cdot\nabla\varphi)+\nabla^{\perp}\Delta_{2D}^{-1}(\nabla^{\perp}\varphi\cdot u).
\end{equation}
 From this, we see that if $u$ is divergence-free and sufficiently integrable with $\Omega_{0}\Subset\{x:\varphi(x)=1\}$, then each component of $h_{\varphi}(u)$ solves Laplace's equation on $\Omega_{0}$ and $h_{\varphi}(u)$ is divergence-free on $\Omega_{0}$. As observed in \cite{K23}, in this case one also has the estimate
 \begin{equation}\label{eq:hestimate}
 \|\nabla^{k} h_{\varphi}(u)\|_{L^{\infty}(\Omega_0)}\lesssim_{k} (\textrm{dist}(\Omega_0, \supp\nabla\varphi))^{-(k+1)}\|u\|_{L^{1}(\Omega)}\|\nabla\varphi\|_{L^{\infty}(\Omega)}
 \end{equation}
  for $k=0,1\ldots.$
 \begin{remark}[Interpretation as a large scale approximation]\label{rmk:largescaleapprox}
 
 For $u$ being a\\ divergence-free function, from \eqref{eq:hestimate} it is reasonable to also interpret $\mathbb{P}_{\varphi} u$ is a \textit{large scale approximation to $u$}. In particular, let $\Omega=B(R)$, with
 $\varphi_{R}\in C^{\infty}_{0}(B(R); [0,1])$ and
 \begin{itemize}
  \item $\varphi_{R}=1$ on $B(\tfrac{3R}{4})$,
  \item $|\nabla \varphi_{R}(x)|\lesssim R^{-1}$ for all $x\in\mathbb{R}^2$.
  \end{itemize}
   Then we see from \eqref{eq:hestimate} that the error satisfies 
   \begin{equation}\label{eq:largescaleapproxerror}
   \begin{split}
   \|\nabla^{k}(\mathbb{P}_{\varphi_{R}}(u)- u)\|_{L^{\infty}(B(\tfrac{R}{2}))}&=\|\nabla^{k}(\mathbb{P}_{\varphi_{R}}(u)-\varphi_{R} u)\|_{L^{\infty}(B(\tfrac{R}{2}))}\\
   &\lesssim_{k} R^{-(k+2)}\|u\|_{L^{1}(B(R))}
   \end{split}
   \end{equation}
   for $k=0,1\ldots.$
 \end{remark}
 Let $\varphi$, $\Omega$, $u$ and $\Omega_0$ be as stated prior to the previous remark, with $u$ being divergence-free. Note that for a vector field $F$, divergence-free vector field $G$ and scalar $\phi$, we have the identities 
 \begin{equation}\label{eq:perpdeltatest}
 \Delta\phi\nabla^{\perp}\cdot F=\nabla^{\perp}\cdot( F\Delta\phi )-\nabla^{\perp}\Delta\phi\cdot F,
 \end{equation}
 \begin{equation}\label{eq:nablaperp}
 \quad \partial_{i}\phi\nabla^{\perp}\cdot F=\nabla^{\perp}\cdot( F\partial_{i}\phi )-\nabla^{\perp}\partial_{i}\phi\cdot F\quad\textrm{for}\,i=1,2,3,
 \end{equation}
 \begin{equation}\label{eq:perpdotlaplace}
 \phi \nabla^{\perp}\cdot\Delta F=\Delta(\phi \nabla^{\perp}\cdot F)+\nabla^{\perp}\cdot( F\Delta\phi )-(\nabla^{\perp}\Delta\phi)\cdot F-2\partial_{i}\,(\nabla^{\perp}\cdot(F\partial_{i}\phi )-\nabla^{\perp}\partial_{i}\phi\cdot F),
 \end{equation}
 \begin{equation}\label{eq:perptestdotconvective}
 \nabla^{\perp}\phi\cdot(G\cdot\nabla G)=\nabla\cdot( G\nabla^{\perp}\phi\cdot G)-G\otimes G:\nabla(\nabla^{\perp}\phi)
 \end{equation}
 and
 \begin{equation}\label{eq:testconvective}
 \phi G\cdot\nabla G =\nabla\cdot((G\otimes G)\phi)-(G\cdot\nabla\phi)G.
 \end{equation}
 
 Following similar arguments to \cite{K23}, in the 2D setting one can use \eqref{eq:locallerayprojectcurl} and \eqref{eq:perpdotlaplace} to show that
 \begin{equation}\label{eq:commutatorlaplace}
 \begin{split}
 &[\Delta,\mathbb{P}_{\varphi}]u:=\Delta\mathbb{P}_{\varphi}(u)-\mathbb{P}_{\varphi}(\Delta u)\\
 &=\nabla^{\perp}\Delta_{2D}^{-1}(u\cdot\nabla^{\perp}\Delta\varphi)+2\nabla^{\perp}\partial_{i}\Big(\nabla^{\perp}\cdot\Delta_{2D}^{-1}(u\partial_{i}\varphi )\Big)\\
 &-2\nabla^{\perp}\partial_{i} \Big(\Delta_{2D}^{-1}(u\cdot \nabla^{\perp}\partial_{i}\varphi )\Big)-\nabla^{\perp}\nabla^{\perp}\cdot \Delta_{2D}^{-1}(u\Delta\varphi).
 \end{split}
 \end{equation}
 Moreover, using \eqref{eq:locallerayprojectiondef} and \eqref{eq:perptestdotconvective}-\eqref{eq:testconvective} we have
 \begin{equation}\label{eq:commutatornonlinear}
 -\varphi u\cdot\nabla u+\mathbb{P}_{\varphi}(u\cdot\nabla u)=\nabla p^{\varphi}_{u\otimes u}+F^{\varphi}_{u\otimes u}\quad\textrm{with}
 \end{equation}
 \begin{equation}\label{eq:localpressureexpression}
 p^{\varphi}_{u\otimes u}:=-\partial_{i}\partial_{j}\Delta_{2D}^{-1}(\varphi u_{i}u_{j})\quad\textrm{and}
 \end{equation}
 \begin{equation}\label{eq:Fdef}
 F^{\varphi}_{u\otimes u}:=\nabla^{\perp}\Delta_{2D}^{-1}(u\otimes u:\nabla(\nabla^{\perp}\varphi))+\nabla(\nabla\cdot(\Delta_{2D}^{-1}(u(u\cdot\nabla\varphi))))-\nabla^{\perp}\nabla\cdot\Delta_{2D}^{-1}( u\nabla^{\perp}\varphi\cdot u).
 \end{equation}
 As observed in \cite{K23}, we can argue similarly as for $h_{\varphi}(u)$ in \eqref{eq:hestimate} to show for $k=0,1\ldots$
\begin{equation}\label{eq:commlaplaceestimate}
 \|\nabla^{k} [\Delta,\mathbb{P}_{\varphi}]u\|_{L^{\infty}(\Omega_0)}\lesssim_{k,\Omega_0,\varphi} \|u\|_{L^{1}(\Omega)}\quad\textrm{and}
 \end{equation}
 \begin{equation}\label{eq:nonlincommutator}
 \begin{split}
 \|\nabla^{k} F^{\varphi}_{u\otimes u}\|_{L^{\infty}(\Omega_0)}&\lesssim_{k,\Omega_0,\varphi} (\|\nabla^2\varphi\|_{L^{\infty}(\Omega)}+\|\nabla\varphi\|_{L^{\infty}(\Omega)})\|u\otimes u\|_{L^{1}(\Omega)}\\
 &\lesssim (\|\nabla^2\varphi\|_{L^{\infty}(\Omega)}+\|\nabla\varphi\|_{L^{\infty}(\Omega)})\|u\|_{L^{2}(\Omega)}^2.
 \end{split}
 \end{equation}
 Furthermore, Calder\'{o}n-Zygmund estimates give that for $1<q<\infty$
 \begin{equation}\label{eq:localpressureest}
 \|p^{\varphi}_{u\otimes u}\|_{L^{q}(\mathbb{R}^2)}\lesssim_{q}\|\varphi u\otimes u\|_{L^q(\mathbb{R}^2)}\lesssim\|\varphi |u|^2\|_{L^q(\mathbb{R}^2)}\lesssim_{q}\|u\|^{2}_{L^{2q}(\supp\varphi)}.
 \end{equation}
Using \eqref{eq:commutatorlaplace}-\eqref{eq:commutatornonlinear}, we see that  if $u^{\nu}:\Omega\times (0,T)\rightarrow \mathbb{R}^2$ and $p^{\nu}:\Omega\times (0,T)\rightarrow\mathbb{R}$ are sufficiently smooth and solve the 2D Navier-Stokes equations, then $\mathbb{P}_{\varphi}(u^{\nu})$ satisfies 
 \begin{equation}\label{eq:localLerayeqn}
 \begin{split}
 &\partial_{t}\mathbb{P}_{\varphi}(u^{\nu})-\nu\Delta \mathbb{P}_{\varphi}(u^{\nu})+u^{\nu}\cdot \nabla u^{\nu}+\nabla p^{\varphi}_{u^{\nu}\otimes u^{\nu}}+ F^{\varphi}_{u^{\nu}\otimes u^{\nu}}+\nu [\Delta,\mathbb{P}_{\varphi}]u^{\nu}=0,\\& \nabla\cdot \mathbb{P}_{\varphi}(u^{\nu})=0\quad \textrm{in}\quad \Omega_0\times (0,T).
 \end{split}
 \end{equation}
 
 Now we can define the local energy balance for the large-scale approximation for a weak solution of 2D Euler.
 \begin{definition}\label{def:locallerayenergybalance}[Local energy balance for the 2D Euler large-scale approximation]
 Let $\Omega\subset\mathbb{R}^2$ be a bounded domain. let
$u^{E}\in L^{3}(0,T;L^{3}(\Omega))$ is a weak solution of the 2D Euler equation in $\Omega\times (0,T)$.
Let $\varphi\in C^{\infty}_{0}(\Omega; [0,1])$ with $\Omega_{0}\Subset \{x:\varphi(x)=1\}.$ We say $\mathbb{P}_{\varphi}(u^{E})$ satisfies the local energy balance on $\Omega_0\times (0,T)$ if
\begin{equation}\label{eq:localenergybalanceeq}
\begin{split}
&\int\limits_{0}^{T}\int\limits_{\Omega_0} -|\mathbb{P}_{\varphi}(u^{E})|^2\partial_{t}\phi\, dxdt=\int\limits_{0}^{T}\int\limits_{\Omega_0} (|u^{E}|^2 u^{E}+2p_{ u^{E}\otimes u^{E}}^{\varphi}\mathbb{P}_{\varphi}(u^{E}))\cdot\nabla \phi \,dxdt\\
&-2\int\limits_{0}^{T}\int\limits_{\Omega_0} u^{E}\otimes u^{E}: \nabla(\phi h_{\varphi}(u^{E}))+F^{\varphi}_{u^{E}\otimes u^{E}}\cdot \mathbb{P}_{\varphi}(u^{E})\phi \, dxdt\quad\\
& \forall\,\phi\in C^{\infty}_{0}(\Omega_{0}\times (0,T)).
\end{split}
\end{equation}
 \end{definition}
 \begin{remark}\label{rmk:termswelldefined}[Integrability assumption in local-energy balance for large scale approximation]
 The estimates \eqref{eq:hestimate} and \eqref{eq:nonlincommutator}-\eqref{eq:localpressureest} ensure that each of the terms in \eqref{eq:localenergybalanceeq} is well defined for $u^{E}\in L^{3}(0,T; L^{3}(\Omega))$.
 \end{remark}
 \begin{remark}\label{rmk:relationtolocenergy}[Relation with the local-energy balance for $u^{E}$]
 Let $u^{E}:\mathbb{R}^2\times (0,T)\rightarrow \mathbb{R}^2$ be a weak solution to the 2D Euler equations on $\mathbb{R}^2\times (0,T)$ with $u^{E}\in L^{3}(0,T; L^{3}(\mathbb{R}^2))$ and corresponding pressure $$p_{u^{E}\otimes u^{E}}:=-\partial_{i}\partial_{j}\Delta_{2D}^{-1}( u^{E}_{i}u^{E}_{j}).$$ 
Let $\phi\in C^{\infty}_{0}(\mathbb{R}^2; [0,1])$ be a fixed test function and $\varphi_{R}$ be as in Remark \ref{rmk:largescaleapprox} with $R\geq 1$ such that $\supp(\phi)\subseteq B(\tfrac{R}{2})$.

Then by \eqref{eq:hestimate}, Remark \ref{rmk:largescaleapprox}, \eqref{eq:nonlincommutator} and Calder\'{o}n-Zygmund estimates, one has that as $R\rightarrow\infty$
$$\|h_{\varphi_{R}}(u^{E})\|_{L^{3}(0,T; L^{3}(\supp(\phi)))}\rightarrow 0,\quad \|\nabla h_{\varphi_{R}}(u^{E})\|_{L^{3}(0,T; L^{3}(\supp(\phi)))}\rightarrow 0,$$
$$\|\mathbb{P}_{\varphi_{R}}(u^{E})-u^{E}\|_{L^{3}(0,T; L^{3}(\supp(\phi)))}\rightarrow 0,\quad \|F^{\varphi_{R}}_{u^{E}\otimes u^{E}}\|_{L^{\frac{3}{2}}(0,T; L^{\frac{3}{2}}(\supp(\phi)))}\rightarrow 0$$ and
$$ \|p^{\varphi_{R}}_{u^{E}\otimes u^{E}}-p_{u^{E}\otimes u^{E}}\|_{L^{\frac{3}{2}}(0,T; L^{\frac{3}{2}}(\mathbb{R}^2))}\rightarrow 0.$$
So if $\mathbb{P}_{\varphi_{R}}(u^{E})$ satisfies \eqref{eq:localenergybalanceeq} for all sufficiently large $R$ ($\phi$ being fixed as in this Remark), one can use the above convergences to pass to the limit $R\rightarrow\infty$ to recover the local energy balance for $u^{E}$:
\begin{equation*}
\begin{split}
&\int\limits_{0}^{T}\int\limits_{\mathbb{R}^2} -|u^{E}|^2\partial_{t}\phi\, dxdt=\int\limits_{0}^{T}\int\limits_{\mathbb{R}^2} (|u^{E}|^2 u^{E}+2p_{ u^{E}\otimes u^{E}}u^{E})\cdot\nabla \phi \,dxdt.
\end{split}
\end{equation*}

 \end{remark}
 
\subsection{Proof of Theorem \ref{thm:localenergybalance}}
From Theorem \ref{thm:nolocaanondis}, we see that for all $n=1,2\ldots$ there exists $t^{*}_{n}=t^{*}_{n}(n,\Omega,T,\varepsilon,N)$ such that 
\begin{equation}\label{eq:thm1conclusionvortboundrecall}
 \sup_{\nu>0}\|\omega^{\nu}\|_{L^{\infty}(0,t^{*}_{n}; L^{p}(\Omega(\frac{1}{n})))}<\infty\quad\textrm{and}
 \end{equation}
 \begin{equation}\label{eq:thm1conclusionrecall}
 \lim_{\nu\rightarrow 0}\nu\int\limits_{0}^{t^{*}_{n}}\int\limits_{\Omega(\frac{1}{n})} |\nabla u^{\nu}|^2 dxdt=0.
 \end{equation}
 By identical reasoning to \cite[Lemma 4.1]{SWW24} based on Aubin-Lions compactness arguments, we can select a diagonal subsequence of $u^{\nu}$ (which we still denote by $u^{\nu}$) such that 
 \begin{equation}\label{eq:strongKEconvergence}
 u^{\nu}\rightarrow u^{E}\quad\textrm{in}\quad C([0,t^{*}_{n}]; L^{2}(\Omega(\tfrac{1}{n})))\quad\forall n=1\ldots.
 \end{equation}
\textbf{Strong convergence locally in $L^{3}_{x,t}$}\\
Recall from \ref{eq:L1plusinfinity} that for all $r>0$
\begin{equation}\label{eq:L1plusinfinityrecall}
 \sup_{\nu>0}\|u^{\nu}\|_{L^{1+\varepsilon}(0,T; L^{\infty}(\Omega(r)))}=N(r)<\infty.
 \end{equation}
 By Lebesgue interpolation, \eqref{eq:L1plusinfinityrecall} and H\"{o}lder's inequality, we have that for any $\nu>0$ and $\nu'>0$ that 
 \begin{equation}\label{eq:L3interpolation1}
 \begin{split}
 &\|u^{\nu}-u^{\nu'}\|_{L^{3}(0,t^{*}_{n}; L^{3}(\Omega(\frac{1}{n})))}\\
 &\leq \|u^{\nu}-u^{\nu'}\|_{L^{\infty}(0,t^{*}_{n}; L^{2}(\Omega(\frac{1}{n})))}^{\frac{2}{3}}\|u^{\nu}-u^{\nu'}\|_{L^{1}(0,t^{*}_{n}; L^{\infty}(\Omega(\frac{1}{n})))}^{\frac{1}{3}}\\
 &\leq (2N(\tfrac{1}{n}))^{\frac{1}{3}}(t^{*}_{n})^{\frac{\varepsilon}{3(1+\varepsilon)}}\|u^{\nu}-u^{\nu'}\|_{L^{\infty}(0,t^{*}_{n}; L^{2}(\Omega(\frac{1}{n})))}^{\frac{2}{3}}.
 \end{split}
 \end{equation}
 From this and \eqref{eq:strongKEconvergence}, we infer that
\begin{equation}\label{eq:strongL3convergenceomegan}
 u^{\nu}\rightarrow u^{E}\quad\textrm{in}\quad L^3(0,t^{*}_{n}; L^{3}(\Omega(\tfrac{1}{n})))\quad\forall n=1\ldots.
\end{equation}
Next,  for each $n=1,2\ldots$ $u^{\nu}$ satisfies 
 \begin{equation}\label{eq:2dNSweakformthm2}
 \begin{split}
 &\int\limits_{0}^{t^{*}_{n}}\int\limits_{\Omega(\tfrac{1}{n})} -u^{\nu}\cdot\partial_{t}\theta-\nu u^{\nu}\cdot\Delta\theta-u^{\nu}\otimes u^{\nu}:\nabla \theta\, dxdt=0\\
 &\forall\quad \textrm{divergence-free}\quad  \theta\in C^{\infty}_{0}(\Omega(\tfrac{1}{n})\times (0,t^{*}_{n})).
 \end{split}
 \end{equation}
 From this and \eqref{eq:strongL3convergenceomegan}, we see from taking the limit as $\nu\rightarrow 0$ that $u^{E}$ is a weak solution to the 2D Euler equations on $\Omega(\tfrac{1}{n})\times (0,t^{*}_{n})$ for each $n=1,2\ldots$.\\
 \textbf{Local energy balance}\\
 To summarize, from \eqref{eq:thm1conclusionrecall}, \eqref{eq:strongKEconvergence} and \eqref{eq:strongL3convergenceomegan}, we have that 
 for all $r>0$, there exists $t^{*}_{r}=t^{*}_{r}(r,\Omega,T,\varepsilon,N)\in (0,T]$ such that as $\nu\rightarrow 0$ (along the subsequence) 
 \begin{equation}\label{eq:noanonomegar}
 \nu\int\limits_{0}^{t^{*}_{r}}\int\limits_{\Omega(r)} |\nabla u^{\nu}|^2 dxds\rightarrow 0,
 \end{equation}
 \begin{equation}\label{eq:strongconvergKEomegar}
  u^{\nu}\rightarrow u^{E}\quad\textrm{in}\quad C([0,t^{*}_{r}]; L^{2}(\Omega(r)))\quad\textrm{and}
 \end{equation}
 \begin{equation}\label{eq:strongL3recallomegar}
 u^{\nu}\rightarrow u^{E}\quad\textrm{in}\quad L^{3}(0,t^{*}_{r}; L^{3}(\Omega(r))).
 \end{equation}
   
 Let $\varphi\in C^{\infty}_{0}(\Omega; [0,1])$ be a test function with
 \begin{equation}\label{eq:supportvarphirecall}
 \Omega_{0}\Subset \{x:\varphi(x)=1\}\subset \supp(\varphi)\Subset\Omega(r),
 \end{equation}
 Let us now show that $\mathbb{P}_{\varphi}(u^{E})$ satisfies the local energy balance in $\Omega_{0}\times (0,t^{*}_{r})$. As $\mathbb{P}_{\varphi}(u^{\nu})$ satisfies  \eqref{eq:localLerayeqn} in $\Omega_0\times (0,T)$, we have that for $\phi\in C^{\infty}_{0}(\Omega_{0}\times (0,t^{*}_{r}))$ that
 \begin{equation}\label{eq:localenergybalanceeqNSE}
\begin{split}
&\int\limits_{0}^{t^{*}_{r}}\int\limits_{\Omega_0} -|\mathbb{P}_{\varphi}(u^{\nu})|^2\partial_{t}\phi dxdt=\int\limits_{0}^{t^{*}_{r}}\int\limits_{\Omega_0} (|u^{\nu}|^2 u^{\nu}+2p_{ u^{\nu}\otimes u^{\nu}}^{\varphi}\mathbb{P}_{\varphi}(u^{\nu}))\cdot\nabla \phi dxdt\\
&-2\int\limits_{0}^{t^{*}_{r}}\int\limits_{\Omega_0} u^{\nu}\otimes u^{\nu}: \nabla(\phi h_{\varphi}(u^{\nu}))+F^{\varphi}_{u^{\nu}\otimes u^{\nu}}\cdot \mathbb{P}_{\varphi}(u^{\nu}) \phi dxdt\\
&+\nu \int\limits_{0}^{t^{*}_{r}}\int\limits_{\Omega_0} |\mathbb{P}_{\varphi}(u^{\nu})|^2\Delta\phi-2|\nabla\mathbb{P}_{\varphi} u^{\nu}|^2\phi-2\phi \mathbb{P}_{\varphi}(u^{\nu})\cdot [\Delta, \mathbb{P}_{\varphi}]u^{\nu}dxdt 
\end{split}
\end{equation}
From the estimate \eqref{eq:CZlocalleray}, the energy inequality \eqref{eq:energyequality} and the assumption \ref{eq:L2iduniform}, we have that
\begin{equation}\label{eq:nondissest}
\begin{split}
&\nu\int\limits_{0}^{t^{*}_{r}}\int\limits_{\Omega_0}|\nabla \mathbb{P}_{\varphi} u^{\nu}|^2 dxds\lesssim_{\varphi} \nu\int\limits_{0}^{t^{*}_{r}}\int\limits_{\Omega(r)}|\nabla u^{\nu}|^2 dxds\\
&+\nu\int\limits_{0}^{t^{*}_{r}}\int\limits_{\Omega(r)}| u^{\nu}|^2 dxds\lesssim \nu t^{*}_{r}J^2+\nu\int\limits_{0}^{t^{*}_{r}}\int\limits_{\Omega(r)}|\nabla u^{\nu}|^2 dxds.
\end{split}
\end{equation}
Using \eqref{eq:noanonomegar}, we then have
\begin{equation}\label{eq:noanondisslocalleray}
\lim_{\nu\rightarrow 0}\nu\int\limits_{0}^{t^{*}_{r}}\int\limits_{\Omega_0}|\nabla \mathbb{P}_{\varphi} u^{\nu}|^2 dxds= 0.
\end{equation}
From the estimates \eqref{eq:CZlocalleray0} and \eqref{eq:localpressureest}, together with the convergence \eqref{eq:strongL3recallomegar}, we infer  that as $\nu\rightarrow 0$ (along the subsequence) that
\begin{equation}\label{eq:strongL3localleray}
\mathbb{P}_{\varphi}u^{\nu}\rightarrow \mathbb{P}_{\varphi}u^{E}\quad\textrm{in}\quad L^{3}(0,t^{*}_{r}; L^{3}(\Omega_0))\quad\textrm{and}
\end{equation}
\begin{equation}\label{eq:strongL32pressure}
p^{\varphi}_{u^{\nu}\otimes u^{\nu}}\rightarrow p^{\varphi}_{u^{E}\otimes u^{E}} \quad\textrm{in}\quad L^{\frac{3}{2}}(0,t^{*}_{r}; L^{\frac{3}{2}}(\Omega_0)).
\end{equation}
The estimates \eqref{eq:hestimate} and \eqref{eq:commlaplaceestimate}-\eqref{eq:nonlincommutator}, together with the convergence \eqref{eq:strongconvergKEomegar}, imply that for all $k=0,1\ldots$
\begin{equation}\label{eq:harmonicstrongconvergence}
\nabla^{k} h_{\varphi}(u^{\nu})\rightarrow  \nabla^{k} h_{\varphi}(u^{E})\quad\textrm{in}\quad  L^{\infty}(0,t^{*}_{r}; L^{\infty}(\Omega_0)),
\end{equation}
\begin{equation}\label{eq:laplacecommutatorstrongconvergence}
\nabla^{k} [\Delta,\mathbb{P}_{\varphi}]u^{\nu}\rightarrow  \nabla^{k} [\Delta,\mathbb{P}_{\varphi}]u^{E}\quad\textrm{in}\quad  L^{\infty}(0,t^{*}_{r}; L^{\infty}(\Omega_0))\quad\textrm{and}
\end{equation}
\begin{equation}\label{eq:nonlincommutatorest}
\nabla^{k} F^{\varphi}_{u^{\nu}\otimes u^{\nu}}\rightarrow  \nabla^{k} F^{\varphi}_{u^{E}\otimes u^{E}}\quad\textrm{in}\quad  L^{\infty}(0,t^{*}_{r}; L^{\infty}(\Omega_0)).
\end{equation}
The convergences \eqref{eq:strongL3recallomegar} and \eqref{eq:noanondisslocalleray}-\eqref{eq:nonlincommutatorest} allow us to take the limit in \eqref{eq:localenergybalanceeqNSE} $\nu\rightarrow 0$ (along the subsequence). This gives that 
\begin{equation*}
\begin{split}
&\int\limits_{0}^{t^{*}_{r}}\int\limits_{\Omega_0} -|\mathbb{P}_{\varphi}(u^{E})|^2\partial_{t}\phi dxdt=\int\limits_{0}^{t^{*}_{r}}\int\limits_{\Omega_0} (|u^{E}|^2 u^{E}+2p_{ u^{E}\otimes u^{E}}^{\varphi}\mathbb{P}_{\varphi}(u^{E}))\cdot\nabla \phi dxdt\\
&-2\int\limits_{0}^{t^{*}_{r}}\int\limits_{\Omega_0} u^{E}\otimes u^{E}: \nabla(\phi h_{\varphi}(u^{E}))+F^{\varphi}_{u^{E}\otimes u^{E}}\cdot \mathbb{P}_{\varphi}(u^{E}) \phi dxdt. 
\end{split}
\end{equation*}
Since this holds for arbitrary $\phi\in C^{\infty}_{0}(\Omega_{0}\times (0,t^{*}_{r}))$, we have that $\mathbb{P}_{\varphi}(u^{E})$ satisfies the local energy balance in $\Omega_{0}\times (0,t^{*}_{r})$ in the sense of Definition  \ref{def:locallerayenergybalance}.

\subsection{Proof of Theorem \ref{thm:localenergybalancecompactL3}}

\begin{itemize}
\item[]\!\!\!\!\!\!\!\!\!\!\!\!\!\!\!\!\!\!\!\!\!\!\textbf{Absence of anomalous dissipation on $\Omega(r)\times (\delta,t^*_{r})$}
\end{itemize}

Define
\begin{equation}\label{eq:R0defthm2}
R_0:=\frac{r}{4}.
\end{equation}
From the assumption \ref{eq:L1plusinfinity} and H\"{o}lder's inequality, we have for any $z\in \Omega(r)$ and $0<S\leq T$ that for all $\nu>0$
\begin{equation}\label{eq:L1Linfinitybound1thm3}
\|u^{\nu}\|_{L^{1}(0,S; L^{\infty}(B(z,R_0)))}\leq S^{\frac{\varepsilon}{1+\varepsilon}}\|u^{\nu}\|_{L^{1+\varepsilon}(0,S; L^{\infty}(\Omega(R_0)))}\leq  S^{\frac{\varepsilon}{1+\varepsilon}} N(R_0).
\end{equation}
Define $t^*_{r}=t^*_{r}({r,\Omega,T,\varepsilon,N})\in (0,T]$ by
\begin{equation}\label{eq:trdefthm3}
t^*_{r}:=\min\Big(T,\frac{1}{2}\Big(\frac{R_0}{8N(R_0)}\Big)^{\frac{1+\varepsilon}{\varepsilon}}\Big).
\end{equation}
Then from \eqref{eq:L1Linfinitybound1thm3} we see that for all $\nu>0$
\begin{equation}\label{eq:L1LinfinityR08thm3}
\|u^{\nu}\|_{L^{1}(0,t^*_{r}; L^{\infty}(B(z,R_0)))}<\frac{R_0}{8}.
\end{equation}
Applying Proposition \ref{prop:aprioriest1.1} for $0<\delta\leq t^*_{r} \quad\textrm{and}\quad 0<\nu< (\frac{5R_0}{8})^{2} $, together with using the energy equality and \ref{eq:L2iduniform}, gives for each $z\in \Omega(r)$
\begin{equation}\label{eq:aprioriest1.1thm3}
\begin{split}
&\nu\int\limits_{\delta}^{t^*_{r}} \int\limits_{B(z,\frac{R_0}{8})} |\nabla u^{\nu}|^2 dxds\lesssim\Big(1+\frac{\nu}{R_0^2}+\frac{1}{\delta}+\frac{\nu t^*_{r}}{\delta R_0^2}\Big)^{\frac{1}{2}} J\times\\
 &\Big(\int\limits_{{\delta}}^{t^*_{r}}\dashint_{B(\nu^{\frac{1}{2}})}\int_{B(z,\frac{3 R_0}{8})}|u^{\nu}(x-y,s)-u^{\nu}(x,s)|^2 dxdyds\Big)^{\frac{1}{2}}\\
 &\leq \Big(1+\frac{\nu}{R_0^2}+\frac{1}{\delta}+\frac{\nu t^*_{r}}{\delta R_0^2}\Big)^{\frac{1}{2}} J\times\\
 &\Big(\int\limits_{{\delta}}^{t^*_{r}}\dashint_{B(\nu^{\frac{1}{2}})}\int_{\Omega(\frac{29r}{32})}|u^{\nu}(x-y,s)-u^{\nu}(x,s)|^2 dxdyds\Big)^{\frac{1}{2}}
\end{split}
\end{equation}
Since $\Omega(r)\subset \Omega$ is compact, we let $(z_{i})_{1\leq i\leq Q}$ be in $\Omega(r)$ such that we have the finite covering
$$\Omega(r)\subset\cup_{i=1}^{Q} B(z_{i}, \tfrac{R_0}{8}).$$
Applying \eqref{eq:aprioriest1.1thm3} to each $z_{i}$ and summing over $i=1\ldots Q$ then gives that
\begin{equation}\label{eq:locanomdissfinalestthm3}
\begin{split}
&\nu\int\limits_{\delta}^{t^*_{r}} \int\limits_{\Omega(r)} |\nabla u^{\nu}|^2 dxds\lesssim \Big(1+\frac{\nu}{R_0^2}+\frac{1}{\delta}+\frac{\nu t^*_{r}}{\delta R_0^2}\Big)^{\frac{1}{2}} JQ\times\\
&\Big(\int\limits_{{\delta}}^{t^*_{r}}\dashint_{B(\nu^{\frac{1}{2}})}\int_{\Omega(\frac{29r}{32})}|u^{\nu}(x-y,s)-u^{\nu}(x,s)|^2 dxdyds\Big)^{\frac{1}{2}}. 
\end{split}
\end{equation}
The strong convergence \eqref{eq:strongL3thm3}, together with the triangle inequality and H\"{o}lder's inequality, readily imply that
$$\lim_{\nu\rightarrow 0}\int\limits_{{\delta}}^{t^*_{r}}\dashint_{B(\nu^{\frac{1}{2}})}\int_{\Omega(\frac{29r}{32})}|u^{\nu}(x-y,s)-u^{\nu}(x,s)|^2 dxdyds=0. $$
Hence, from \eqref{eq:locanomdissfinalestthm3} we obtain
that for $0<\delta\leq t^*_{r}$ that
\begin{equation}\label{eq:vanishingdissstep1}
\lim_{\nu\rightarrow 0}\nu\int\limits_{\delta}^{t^*_{r}} \int\limits_{\Omega(r)} |\nabla u^{\nu}|^2 dxds=0.
\end{equation}

Next,  $u^{\nu}$ satisfies 
 \begin{equation}\label{eq:2dNSweakformthm3}
 \begin{split}
 &\int\limits_{0}^{t^*_{r}}\int\limits_{\Omega(r)} -u^{\nu}\cdot\partial_{t}\theta-\nu u^{\nu}\cdot\Delta\theta-u^{\nu}\otimes u^{\nu}:\nabla \theta\, dxdt=0\\
 &\forall\quad \textrm{divergence-free}\quad  \theta\in C^{\infty}_{0}(\Omega(r)\times (0,t^*_{r})).
 \end{split}
 \end{equation}
 From this and \eqref{eq:strongL3thm3}, we see from taking the limit as $\nu\rightarrow 0$ that $u^{E}$ is a weak solution to the 2D Euler equations on $\Omega(r)\times (0,t^*_{r})$.\\
 \textbf{Local energy balance}\\
 To summarize, from \eqref{eq:vanishingdissstep1} and \eqref{eq:strongL3thm3}, we have that there exists $t^*_{r}=t^*_{r}(r,\Omega,T,\varepsilon,N)$ belonging to $(0,T]$ such that as $\nu\rightarrow 0$ 
 \begin{equation}\label{eq:noanonomegarthm3}
 \nu\int\limits_{\delta}^{t^*_{r}}\int\limits_{\Omega(r)} |\nabla u^{\nu}|^2 dxds\rightarrow 0\quad\textrm{for}\quad 0<\delta\leq t^*_{r}\quad\textrm{and}
 \end{equation}
 \begin{equation}\label{eq:strongL3recallomegarthm3}
 u^{\nu}\rightarrow u^{E}\quad\textrm{in}\quad L^{3}(0,t^*_{r}; L^{3}(\Omega(r))).
 \end{equation}
   
 Let $\varphi\in C^{\infty}_{0}(\Omega; [0,1])$ be a test function with
 \begin{equation}\label{eq:supportvarphirecallthm3}
 \Omega_{0}\Subset \{x:\varphi(x)=1\}\subset \supp(\varphi)\Subset\Omega(r),
 \end{equation}
 Let us now show that $\mathbb{P}_{\varphi}(u^{E})$ satisfies the local energy balance in $\Omega_{0}\times (0,t^*_{r})$. As $\mathbb{P}_{\varphi}(u^{\nu})$ satisfies  \eqref{eq:localLerayeqn} in $\Omega_0\times (0,T)$, we have that for $\phi\in C^{\infty}_{0}(\Omega_{0}\times (0,t^*_{r}))$ that
 \begin{equation}\label{eq:localenergybalanceeqNSEthm3}
\begin{split}
&\int\limits_{0}^{t^*_{r}}\int\limits_{\Omega_0} -|\mathbb{P}_{\varphi}(u^{\nu})|^2\partial_{t}\phi dxdt=\int\limits_{0}^{t^*_{r}}\int\limits_{\Omega_0} (|u^{\nu}|^2 u^{\nu}+2p_{ u^{\nu}\otimes u^{\nu}}^{\varphi}\mathbb{P}_{\varphi}(u^{\nu}))\cdot\nabla \phi dxdt\\
&-2\int\limits_{0}^{t^*_{r}}\int\limits_{\Omega_0} u^{\nu}\otimes u^{\nu}: \nabla(\phi h_{\varphi}(u^{\nu}))+F^{\varphi}_{u^{\nu}\otimes u^{\nu}}\cdot \mathbb{P}_{\varphi}(u^{\nu}) \phi dxdt\\
&+\nu \int\limits_{0}^{t^*_{r}}\int\limits_{\Omega_0} |\mathbb{P}_{\varphi}(u^{\nu})|^2\Delta\phi-2|\nabla\mathbb{P}_{\varphi} u^{\nu}|^2\phi-2\phi \mathbb{P}_{\varphi}(u^{\nu})\cdot [\Delta, \mathbb{P}_{\varphi}]u^{\nu}dxdt 
\end{split}
\end{equation}
As $\phi\in C^{\infty}_{0}(\Omega_{0}\times (0,t^*_{r}))$, there exists $0<\delta\leq t^*_{r}$ such that $\phi\in C^{\infty}_{0}(\Omega_{0}\times (\delta,t^*_{r}))$.
From this, the estimate \eqref{eq:CZlocalleray}, the energy inequality \eqref{eq:energyequality} and the assumption \ref{eq:L2iduniform}, we have that
\begin{equation}\label{eq:nondissestthm3}
\begin{split}
&\nu\int\limits_{0}^{t^*_{r}}\int\limits_{\Omega_0}|\nabla \mathbb{P}_{\varphi} u^{\nu}|^2 \phi dxds\lesssim_{\varphi} \nu\int\limits_{\delta}^{t^*_{r}}\int\limits_{\Omega(r)}|\nabla u^{\nu}|^2 dxds\\
&+\nu\int\limits_{\delta}^{t^*_{r}}\int\limits_{\Omega(r)}| u^{\nu}|^2 dxds\lesssim \nu t^*_{r}J^2+\nu\int\limits_{\delta}^{t^*_{r}}\int\limits_{\Omega(r)}|\nabla u^{\nu}|^2 dxds.
\end{split}
\end{equation}
Using \eqref{eq:noanonomegarthm3}, we then have
\begin{equation}\label{eq:noanondisslocalleraythm3}
\lim_{\nu\rightarrow 0}\nu\int\limits_{0}^{t^*_{r}}\int\limits_{\Omega_0}|\nabla \mathbb{P}_{\varphi} u^{\nu}|^2 \phi dxds= 0.
\end{equation}
From the estimates \eqref{eq:CZlocalleray0} and \eqref{eq:localpressureest}, together with the convergence \eqref{eq:strongL3recallomegarthm3}, we infer  that as $\nu\rightarrow 0$ (along the subsequence) that
\begin{equation}\label{eq:strongL3localleraythm3}
\mathbb{P}_{\varphi}u^{\nu}\rightarrow \mathbb{P}_{\varphi}u^{E}\quad\textrm{in}\quad L^{3}(0,t^*_{r}; L^{3}(\Omega_0))\quad\textrm{and}
\end{equation}
\begin{equation}\label{eq:strongL32pressurethm3}
p^{\varphi}_{u^{\nu}\otimes u^{\nu}}\rightarrow p^{\varphi}_{u^{E}\otimes u^{E}} \quad\textrm{in}\quad L^{\frac{3}{2}}(0,t^*_{r}; L^{\frac{3}{2}}(\Omega_0)).
\end{equation}
The estimates \eqref{eq:hestimate} and \eqref{eq:commlaplaceestimate}-\eqref{eq:nonlincommutator}, together with the convergence \eqref{eq:strongL3recallomegarthm3}, imply that for all $k=0,1\ldots$
\begin{equation}\label{eq:harmonicstrongconvergencethm3}
\nabla^{k} h_{\varphi}(u^{\nu})\rightarrow  \nabla^{k} h_{\varphi}(u^{E})\quad\textrm{in}\quad  L^{3}(0,t^*_{r}; L^{\infty}(\Omega_0)),
\end{equation}
\begin{equation}\label{eq:laplacecommutatorstrongconvergencethm3}
\nabla^{k} [\Delta,\mathbb{P}_{\varphi}]u^{\nu}\rightarrow  \nabla^{k} [\Delta,\mathbb{P}_{\varphi}]u^{E}\quad\textrm{in}\quad  L^{3}(0,t^*_{r}; L^{\infty}(\Omega_0))\quad\textrm{and}
\end{equation}
\begin{equation}\label{eq:nonlincommutatorestthm3}
\nabla^{k} F^{\varphi}_{u^{\nu}\otimes u^{\nu}}\rightarrow  \nabla^{k} F^{\varphi}_{u^{E}\otimes u^{E}}\quad\textrm{in}\quad  L^{\frac{3}{2}}(0,t^*_{r}; L^{\infty}(\Omega_0)).
\end{equation}
The convergences \eqref{eq:strongL3recallomegarthm3} and \eqref{eq:noanondisslocalleraythm3}-\eqref{eq:nonlincommutatorestthm3} allow us to take the limit in \eqref{eq:localenergybalanceeqNSEthm3} $\nu\rightarrow 0$ (along the subsequence). This gives that 
\begin{equation*}
\begin{split}
&\int\limits_{0}^{t^*_{r}}\int\limits_{\Omega_0} -|\mathbb{P}_{\varphi}(u^{E})|^2\partial_{t}\phi dxdt=\int\limits_{0}^{t^*_{r}}\int\limits_{\Omega_0} (|u^{E}|^2 u^{E}+2p_{ u^{E}\otimes u^{E}}^{\varphi}\mathbb{P}_{\varphi}(u^{E}))\cdot\nabla \phi dxdt\\
&-2\int\limits_{0}^{t^*_{r}}\int\limits_{\Omega_0} u^{E}\otimes u^{E}: \nabla(\phi h_{\varphi}(u^{E}))+F^{\varphi}_{u^{E}\otimes u^{E}}\cdot \mathbb{P}_{\varphi}(u^{E}) \phi dxdt. 
\end{split}
\end{equation*}
Since this holds for arbitrary $\phi\in C^{\infty}_{0}(\Omega_{0}\times (0,t^*_{r}))$, we have that $\mathbb{P}_{\varphi}(u^{E})$ satisfies the local energy balance in $\Omega_{0}\times (0,t^*_{r})$ in the sense of Definition  \ref{def:locallerayenergybalance}.
\section*{Acknowledgments}
T.B. wishes to thank Universidade Federal do Rio de Janeiro for its gracious hospitality.
T.B. is supported by an EPSRC New Investigator Award UKRI096 ‘Dynamics and regularity criteria for nonlinear incompressible partial differential equations' and was partly supported by start-up funds from the University of Bath. H.J.N.L. and M.C.L.F are grateful for the kind hospitality of the University of Bath. M.C.L.F. wishes to thank the INSMI-CNRS for the invited CNRS research position for the thematic trimester program on Mathematical developments in Geophysical Fluid Dynamics  from April 13 to July 10, 2026 at the Institut Henri  Poincaré, Paris. H.J.N.L. thanks ENS-PSL for the Invited Professor appointment for the same program. Both M.C.L.F. and H.J.N.L. thank the support of the Institut Henri Poincar\'e (UAR 389 CNRS-Sorbonne Universit\'e), and the LabEx CARMIN (ANR-10-LABX-59-01). M.C.L.F. is partially supported by CNPq through grant 304990/2022-1.
H.J.N.L. is supported in part by CNPq through grant 305309/2022-6 and by FAPERJ through grant E-26/200.273/2026.

\bibliography{AbsenceLocAnomDissip2D_Toby-Lopess_Ver03}
\bibliographystyle{abbrv}

\end{document}